\documentclass[a4paper,reqno,11pt]{amsart}
\usepackage{amsfonts,amsmath,amsthm,amssymb,stmaryrd}
\newtheorem{theorem}{Theorem}[section]
\newtheorem{lemma}[theorem]{Lemma}

\newtheorem{Remark}[theorem]{Remark}
\newtheorem{Corollary}[theorem]{Corollary}
\numberwithin{equation}{section}
\allowdisplaybreaks

\usepackage[left=1 in, right=1 in,top=1 in, bottom=1 in]{geometry}

\DeclareMathOperator{\diverge}{div} 
 \providecommand{\norm}[1]{\left\Vert#1\right\Vert}

\def\r3{\mathbb{R}^3}

\begin{document}
\title[Decay of NSP]{Decay of the Navier-Stokes-Poisson equations}

\author{Yanjin Wang}
\address{
School of Mathematical Sciences\\
Xiamen University\\
Xiamen, Fujian 361005, China}
\email[Y. J. Wang]{yanjin$\_$wang@xmu.edu.cn}

\thanks{This work was partially supported by National Natural Science Foundation of China-NSAF (No. 10976026)}

\begin{abstract}
We establish the time decay rates of the solution to the Cauchy problem for the compressible Navier-Stokes-Poisson system via a refined pure energy method. In particular, the optimal decay rates of the higher-order spatial derivatives of the solution are obtained. The $\Dot{H}^{-s}$($0\le s<3/2$) negative Sobolev norms are shown to be preserved along time evolution and enhance the decay rates. As a corollary, we also obtain the usual $L^p$--$L^2$($1<p\le 2$) type of the optimal decay rates. Compared to the compressible Navier-Stokes system and the compressible irrotational Euler-Poisson system, our results imply that both the dispersion effect of the electric field and the viscous dissipation contribute to enhance the decay rate of the density. Our proof is based on  a family of scaled energy estimates with minimum derivative counts and interpolations among them without linear decay analysis.
\end{abstract}

\maketitle

%\tableofcontents

%%%%%%%%%%%%%%%%%%%%%%%%%%%%%%%%%%%%%%%%%%%%%%%
\section{Introduction}
%%%%%%%%%%%%%%%%%%%%%%%%%%%%%%%%%%%%%%%%%%%%%%%
The dynamic of charged particles of one carrier type (e.g., electrons) in the absence of magnetic effects can be described by the compressible (unipolar) Navier-Stokes-Poisson equations (NSP):
\begin{equation}\label{1NS}
\left\{\begin{array}{lll}
\partial_t\rho+{\rm div}(\rho  u )=0
\\\partial_t(\rho  u) +{\rm div}(\rho  u  \otimes u )+\nabla p(\rho)-\mu\Delta u -(\mu+\lambda)\nabla{\rm div} u =\rho\nabla\Phi
\\\Delta\Phi=\rho-\bar{\rho}
\\(\rho,u)|_{t=0}=(\rho_0,  u _0).
\end{array}\right.
\end{equation}
Here  $\rho(t,x)\ge 0,  u(t,x)$ represent the density and the velocity functions of the electrons respectively, at time $t\ge0$ and position $x\in \mathbb{R}^3$. The self-consistent electric potential $\Phi=\Phi(t,x)$ is coupled with the density through the Poisson equation.  The pressure $p =p (\rho )$ is a smooth function  with
$p '(\rho )>0$ for $\rho >0$. We assume that the constant viscosity coefficients $\mu$ and $\lambda$ satisfy the usual physical conditions
\begin{equation}\label{1viscosity}
\mu>0,\quad \lambda+\frac{2}{3}\mu\ge 0.
\end{equation}
In the motion of the fluid, due to the greater inertia the ions merely provide a constant charged background $\bar{\rho}>0$. For simplicity, we take $\bar{\rho}=1$ and assume that $p'(1)=1$.

The main purpose of this paper is to investigate the influence of the
electric field on the time decay rates of the solution compared to the compressible Navier-Stokes equations (NS).
We first review some previous works on the global existence of the
solutions to the NSP system. For the pressure $p(\rho)=\rho^\gamma$ with $\gamma> {3}/{2}$, the global existence of weak solutions was obtained by \cite{KS2008} when the spatial dimension is three. Later this result  was extended by \cite{TW} to the case $\gamma>1$ when the dimension is two, where the authors introduced an idea to overcome the new difficulty caused by that the Poisson term $\rho\nabla\Phi$ may not be integrable when the $\gamma$ is close to one. The constraint of $\gamma$ is somewhat optimal in the sense of the well-known framework of weak solutions to the NS system \cite{FNP,L}. The global existence of small strong solutions in $H^N$ Sobolev spaces was shown in \cite{LMZ} in  the framework of Matsumura and Nishida \cite{MN1}, while global existence of small solutions in some Besov spaces was obtained in \cite{HL0,TWu}.

The convergence rate of the solutions towards the steady state has been an important problem in the PDE theory. The decay rate of solutions to the NS system has been investigated extensively since the works \cite{MN1,MN2,MN3}, see for instance \cite{M,MN2,P,KS1,K1,KK1,KK2,DUYZ1,DLUY,D1,D2,HZ1,HZ2,LW,WY} and the references therein. When the initial perturbation $\rho_0-1,u_0\in L^p\cap H^N$ with $p\in [1,2]$ (Indeed, in those references $p$ is  near 1 and $N\ge 3$ is a large enough integer for the nonlinear system.), the $L^2$ optimal decay rate of the solution to the NS system is
\begin{equation}
\norm{(\rho-1, u) (t)}_{L^2}
\lesssim(1+t)^{-\frac{3}{2}\left(\frac{1}{p}-\frac{1}{2}\right)}.
\end{equation}
Recently, the decay rate of solutions to the NSP system was investigated in \cite{LMZ,ZLZ,WW,HL}. It is observed that the electric field has significant effects on the large time behavior
of the solution. When the initial perturbation $\rho_0-1,u_0\in L^p\cap H^N$ with $p\in [1,2]$, then the $L^2$ optimal decay rate of the solution to the  NSP system is
\begin{equation}\label{LNSP}
\norm{(\rho-1)(t)}_{L^2}
\lesssim(1+t)^{-\frac{3}{2}\left(\frac{1}{p}-\frac{1}{2}\right)}\ \text{ and }\norm{u  (t)}_{L^2}\lesssim(1+t)^{-\frac{3}{2}\left(\frac{1}{p}-\frac{1}{2}\right)+\frac{1}{2}}.
\end{equation}
This implies that the presence of the electric field slows down the decay rate of the velocity of the NSP system with the factor $1/2$ compared to the NS system. The proof is based on that the NSP system can be transformed into the NS system with a non-local force term
\begin{equation}\label{1NS'}
\left\{\begin{array}{lll}
\partial_t\rho+{\rm div}(\rho  u )=0
\\\partial_t(\rho  u) +{\rm div}(\rho  u  \otimes u )+\nabla p(\rho)-\mu\Delta u -(\mu+\lambda)\nabla{\rm div} u =\rho\nabla\Delta^{-1}(\rho-1)
\\(\rho,u)|_{t=0}=(\rho_0,  u _0).
\end{array}\right.
\end{equation}
 By the detailed analysis of the Fourier transform of the Green function for the linear homogeneous system of \eqref{1NS'}, we may have the following approximation for the Fourier transform of the solution, by refining the estimates (3.3)--(3.4) of \cite{LMZ},
\begin{equation}\label{fourier 1}
 \hat{\varrho}(\xi,t)
  \sim
 \left\{\begin{array}{lll}\displaystyle
 O(1)e^{-(\mu+\frac{1}{2}\lambda)|\xi|^2t}\left(|\hat{\varrho}_0|+|\xi||\hat{u}_0|\right),
   &\displaystyle |\xi|\le \eta,\smallskip\smallskip
   \\
\displaystyle O(1)e^{-R_0t}(|\hat{\varrho}_0|+|\hat{u}_0|),&\displaystyle |\xi|\ge \eta,
\end{array}
\right.
\end{equation}
and
\begin{equation}\label{fourier 2}
 \hat{u}(\xi,t)
  \sim
 \left\{\begin{array}{lll}
\displaystyle O(1)e^{-\mu|\xi|^2t}\left(|\xi|^{-1}{|\hat{\varrho}_0|} +|\hat{u}_0|\right),
   &\displaystyle |\xi|\le \eta,\smallskip\smallskip
   \\
\displaystyle O(1)e^{-R_0t}(|\hat{\varrho}_0|+|\hat{u}_0|),&\displaystyle |\xi|\ge \eta.
\end{array}
\right.
\end{equation}
 Hereafter we may sometimes write $\varrho=\rho-1$. $R_0>0$ is a constant and $\eta>0$ is a small but fixed constant. Then the linear optimal decay rate \eqref{LNSP} follows if $\varrho_0,u_0\in L^p$ with $p\in [1,2]$.

However, in this paper we will give a different (contrary) comprehension of the effect of the electric field on the time decay rates of the solution. The key motivation is that if we take $p=2$ in the time decay rate \eqref{LNSP}, then we should get that the $L^2$ norm of $u$ grows in time at the rate $(1+t)^{{1}/{2}}$! This seems unsuitable since the NSP system is a dissipative system. The reason why this happened is that to derive \eqref{LNSP} with $p=2$ it only assume that $\varrho_0,u_0\in L^2$, but from the point of view of the energy structure of the NSP system it is natural to assume that $\nabla\Phi_0\in L^2$. The linear energy identity of the perturbation form of \eqref{1NS} reads as
\begin{equation}
\frac{1}{2}\frac{d}{dt}\int_{\mathbb{R}^3} |\rho-1|^2+|u|^2+ |\nabla\Phi|^2\,dx +\int_{\mathbb{R}^3}\mu|\nabla
u|^2+(\mu+{\lambda})|{\rm div}u|^2\,dx=0.
\end{equation}
By the Poisson equation, the condition $\nabla\Phi_0\in L^2$ is equivalent to that $\Lambda^{-1}\varrho_0\in L^2$. Motivated by this, we instead assume that $\Lambda^{-1}\varrho_0,u_0\in L^p$ with $p\in [1,2]$, then by \eqref{fourier 1}--\eqref{fourier 2}, we have the following $L^2$ optimal decay rates for the linear NSP system:
\begin{equation}\label{LNSP'}
\norm{(\rho-1)(t)}_{L^2}
\lesssim(1+t)^{-\frac{3}{2}\left(\frac{1}{p}-\frac{1}{2}\right)-\frac{1}{2}}\ \text{ and }\norm{u  (t)}_{L^2}\lesssim(1+t)^{-\frac{3}{2}\left(\frac{1}{p}-\frac{1}{2}\right)}.
\end{equation}
In this sense, the electric field does not slow down but rather enhances the time decay rate of the density with the factor $1/2$! This can be understood well from the physical point of view since we get an additional dispersive effect from the repulsive electric force. This is also consistent with \cite{G} in the study of the compressible Euler-Poisson equations.

In the usual $L^{p}$--$L^{2}$ approach of studying the optimal decay rates of the solutions, it is difficult to show that the $L^{p}$ norm of the solution can be
preserved along time evolution. Motivated by \cite{GT}, using a negative Sobolev space $\dot{H}^{-s}$ ($s\ge 0$) to replace $L^{p}$
space, we developed in \cite{GW} a general energy method of using a family of scaled energy
estimates with minimum derivative counts and interpolations among them (but without linear decay analysis) to prove the optimal decay rate of the dissipative equations in the whole space. An important feature is that the $\dot{H}^{-s}$ norm of the solution is preserved along time evolution. The method was applied to classical
examples in \cite{GW} such as the heat equation, the compressible Navier-Stokes equations and the Boltzmann
equation. In this paper, we will apply this energy method to prove the $L^2$ optimal decay rate of the solution to the NSP system \eqref{1NS}.
\smallskip

\noindent \textbf{Notation.} In this paper, $\nabla ^{\ell }$ with an
integer $\ell \geq 0$ stands for the usual any spatial derivatives of order $%
\ell $. When $\ell <0$ or $\ell $ is not a positive integer, $\nabla ^{\ell }
$ stands for $\Lambda ^{\ell }$ defined by \eqref{1Lambdas}. We use $\dot{H}%
^{s}(\mathbb{R}^{3}),s\in \mathbb{R}$ to denote the homogeneous Sobolev
spaces on $\mathbb{R}^{3}$ with norm $\norm{\cdot}_{\dot{H}^{s}}$ defined by %
\eqref{1snorm}, and we use $H^{s}(\mathbb{R}^{3})$ to denote the usual
Sobolev spaces with norm $\norm{\cdot}_{H^{s}}$ and $L^{p}(\mathbb{R}^{3}),1\leq p\leq
\infty $ to denote the usual $L^{p}$ spaces with norm $%
\norm{\cdot}_{L^{p}}$. We will employ the notation $a\lesssim b$ to mean
that $a\leq Cb$ for a universal constant $C>0$ that only depends on the parameters coming from the problem, and the indexes $N$ and $s$ coming from the regularity on the data. We also use $C_0$ for a positive constant depending additionally on the initial data.
\smallskip

Our main results are stated in the following theorem.
\begin{theorem}\label{1mainth}
Assume that $\rho_{0 }-1, u_{0 },\nabla\Phi_0\in H^{N}$ for an integer $N\ge 3$ and
\begin{equation}\label{neutral condition}
\int_{\r3} (\rho_{0 }-1)\,dx=0\ (\text{neutrality}).
\end{equation}
 Then there exists a constant $\delta_0$ such that if
\begin{equation}\label{1Hn/2}
\norm{\rho_{0 }-1}_{H^{3}}+\norm{u_{0 }}_{H^{3}}+\norm{\nabla\Phi_0}_{H^{3}}\le \delta_0,
\end{equation}
then the problem \eqref{1NS} admits a unique global solution $(\rho,u,\nabla\Phi)$ satisfying that for all $t\ge 0$,
\begin{equation}\label{1HNbound}
\begin{split}
&\norm{ (\rho -1)(t)}_{H^{N}}^2+\norm{ u (t)}_{H^{N}}^2+\norm{ \nabla\Phi(t)}_{H^{N}}^2
+\int_0^t \norm{(\rho -1)(\tau)}_{H^{N}}^2+\norm{\nabla
u (\tau)}_{H^{N}}^2+\norm{\nabla \nabla\Phi(\tau)}_{H^{N}}^2d\tau
\\&\quad\le C\left(\norm{ \rho_{0 }-1}_{H^{N}}^2+\norm{ u_{0 }}_{H^{N}}^2+\norm{ \nabla\Phi_0}_{H^{N}}^2\right).
\end{split}
\end{equation}

If further, $\rho_{0 }-1, u_{0 },\nabla\Phi_0\in \dot{H}^{-s}$ for some $s\in [0,3/2)$, then for all $t\ge 0$,
\begin{equation}\label{1H-sbound}
\norm{ (\rho -1)(t)}_{\dot{H}^{-s}}^2+\norm{ u(t)}_{\dot{H}^{-s}}^2+\norm{ \nabla\Phi(t)}_{\dot{H}^{-s}}^2\le C_0,
\end{equation}
and the following decay results hold:
\begin{equation}\label{1decay}
\norm{\nabla^\ell (\rho -1)(t)}_{H^{N-\ell}}+\norm{\nabla^\ell u(t)}_{H^{N-\ell}}+\norm{\nabla^\ell \nabla\Phi(t)}_{H^{N-\ell}} \le
C_0(1+t)^{-\frac{\ell+s}{2}}\ \hbox{ for }\ell=0,\dots, N-1,
\end{equation}
and
\begin{equation}\label{1decay2}
\norm{\nabla^\ell (\rho -1)(t)}_{L^2}\le C_0(1+t)^{-\frac{\ell+s+1}{2}}\ \hbox{ for }\ell=0,\dots, N-2.
\end{equation}
\end{theorem}

Note that the Hardy-Littlewood-Sobolev theorem  (cf. Lemma \ref{1Riesz}) implies that for $p\in (1,2]$, $L^p\subset \Dot{H}^{-s}$ with $s=3(\frac{1}{p}-\frac{1}{2})\in[0,3/2)$. Then by Theorem \ref{1mainth}, we have the following corollary of the usual $L^p$--$L^2$ type of the optimal decay results:
\begin{Corollary}\label{2mainth}
Under the assumptions of Theorem \ref{1mainth} except that we replace the $\Dot{H}^{-s}$ assumption  by that $\rho_{0 }-1, u_{0 },\nabla\Phi_0\in L^p$ for some $p\in (1,2]$, then the following decay results hold:
\begin{equation}\label{1decay111}
\norm{\nabla^\ell (\rho -1)(t)}_{H^{N-\ell}}+\norm{\nabla^\ell u(t)}_{H^{N-\ell}}+\norm{\nabla^\ell \nabla\Phi(t)}_{H^{N+1-\ell}} \le
C_0(1+t)^{-\sigma_{p,\ell}}\ \hbox{ for }\ell=0,\dots,N-1,
\end{equation}
and
\begin{equation}\label{1decay2222}
\norm{\nabla^\ell (\rho -1)(t)}_{L^2}\le C_0(1+t)^{-\left(\sigma_{p,\ell}+\frac{1}{2}\right)}\ \hbox{ for }\ell=0,\dots,N-2.
\end{equation}
Here the number $\sigma_{p,\ell}$ is defined by
\begin{equation}
\sigma_{p,\ell}:=\frac{3}{2}\left(\frac{1}{p}-\frac{1}{2}\right)+\frac{\ell}{2}.
\end{equation}

\end{Corollary}

The followings are several remarks for Theorem \ref{1mainth} and Corollary \ref{2mainth}.

\begin{Remark}
We remark that by the Poisson equation,
\begin{equation}
\norm{\rho_{0 }-1}_{H^{N}\cap \dot{H}^{-s}}+\norm{\nabla\Phi_0}_{H^{N}\cap \dot{H}^{-s}}\sim \norm{\rho_{0 }-1}_{H^{N}\cap \dot{H}^{-s} }+\norm{\Lambda^{-1}\rho_0}_{\dot{H}^{-s}},\ s\ge 0,
\end{equation}
and
\begin{equation}
\norm{\rho_{0 }-1}_{H^{N}\cap L^p}+\norm{\nabla\Phi_0}_{H^{N}\cap L^p}\sim \norm{\rho_{0 }-1}_{H^{N}\cap L^p }+\norm{\Lambda^{-1}\rho_0}_{L^p},\ p\in [1,2].
\end{equation}
Compared to the study of the NS system, such norms on $\Lambda^{-1}\varrho_0$ are additionally required. But they can be achieved by the natural neutral condition \eqref{neutral condition}, see the proof in pp. 263--264 of \cite{G}.
\end{Remark}

\begin{Remark}
Notice that for the global existence of the solution we only assume that the $H^3$ norm of initial data is small, while the higher-order Sobolev norms can be arbitrarily large. This is an improvement of \cite{GW} where we required the smallness of the $H^{ [\frac{N}{2}]+2}$ norm of initial data. Also notice that we do not  assume that $\dot{H}^{-s}$ or $L^p$ norm of initial data is small.
\end{Remark}

\begin{Remark}
Notice that both $\dot{H}^{-s}$ and $L^p$ norms enhance the decay rate of the solution. The constraint $s<3/2$ in Theorem \ref{1mainth} comes
from applying Lemma \ref{1Riesz} to estimate the nonlinear terms when doing
the negative Sobolev estimates via $\Lambda ^{-s}$. For $s\geq 3/2$, the
nonlinear estimates would not work. This in turn restricts $p>1$ in Corollary \ref{2mainth} by our method. However, the constraint $p>1$ also seems necessary for the usual $L^p$--$L^2$ approach. In that approach, one should need to estimate the $L^p$ norm of the $\Lambda^{-1}$ acting on the nonlinear terms by using the linear decay rate \eqref{LNSP'}. This requires $p>1$ for applying Lemma \ref{1Riesz}.
\end{Remark}

\begin{Remark}
Note that the $L^2$ optimal decay rate of the higher-order spatial derivatives of the solution are obtained. Then the general optimal $L^q$ decay rates of the solution follow by  the Sobolev interpolation (cf. Lemma \ref{1interpolation}). For instance, it follows from \eqref{1decay111}--\eqref{1decay2222} that
\begin{equation}  \label{Linfty decay}
\norm{\varrho(t)}_{L^\infty}\le C\norm{\varrho(t)}_{L^2}^{\frac{1}{4}}\norm{\nabla^2
\varrho(t)}_{L^2}^{\frac{3}{4}}\le C_0(1+t)^{-\frac{3}{2p}-\frac{1}{2}}\text{ and }
\norm{\left(u,\nabla\Phi\right)(t)}_{L^\infty} \le C_0(1+t)^{-\frac{3}{2p}}.
\end{equation}
 We remark that Corollary \ref{2mainth} not only provides an alternative approach to derive the $L^p$--$L^2$ type of the optimal decay results but also improves the results in the usual $L^p$--$L^2$ approach, where one seems needing to assume that $p$ is near $1$ for the nonlinear system and the optimal decay rate of the higher-order (greater than 3) spatial derivatives of the solution is not clear.
\end{Remark}

We will prove Theorem \ref{1mainth}  by the energy method that we recently developed in \cite{GW}. As there, we may use the linear heat equation to illustrate the main idea of this approach in advance. Let $u(t)$ be the solution to the heat equation
\begin{equation}\label{heat equation}
\left\{\begin{array}{lll}
\partial_tu -\Delta u=0\ \text{ in }\mathbb{R}^3
\\u|_{t=0}=u_0,
\end{array}\right.
\end{equation}
Let $-s\le \ell\le N$. The standard energy identity of \eqref{heat equation} is
\begin{equation}\label{heat 1}
\frac{1}{2}\frac{d}{dt}\norm{\nabla^\ell u}_{L^2}^2+\norm{\nabla^{\ell+1} u}_{L^2}^2=0.
\end{equation}
Integrating the above in time, we obtain
\begin{equation}\label{heat 2}
\norm{\nabla^\ell u(t)}_{L^2}^2\le \norm{\nabla^\ell u_0}_{L^2}^2.
\end{equation}
Note that the energy in \eqref{heat 1} is not bounded by the corresponding dissipation. But the crucial observation is that by the Sobolev interpolation the dissipation still can give some control on the energy: for $-s<\ell\le N$, by Lemma \ref{1-sinte}, we interpolate
to get
\begin{equation}\label{heat 3}
\norm{\nabla^\ell u(t)}_{L^2} \le \norm{\Lambda^{-s} u(t)}_{L^2}^\frac{1}{\ell+1+s}\norm{\nabla^{\ell+1} u(t)}_{L^2}^\frac{\ell+s}{\ell+1+s}.
\end{equation}
Combining \eqref{heat 3} and \eqref{heat 2} (with $\ell=-s$), we obtain
\begin{equation}\label{heat 4}
\norm{\nabla^{\ell+1} u(t)}_{L^2}\ge \norm{\Lambda^{-s} u_0}_{L^2}^{-\frac{1}{\ell+s}}\norm{\nabla^\ell u(t)}_{L^2}^{1+\frac{1}{\ell+s}}.
\end{equation}
Plugging \eqref{heat 4} into \eqref{heat 1}, we deduce that there exists a constant $C_0>0$ such
that
\begin{equation}\label{heat 5}
\frac{d}{dt}\norm{\nabla^\ell u}_{L^2}^2+C_0\left(\norm{\nabla^\ell u}_{L^2}^2\right)^{1+\frac{1}{\ell+s}}\le 0.
\end{equation}
Solving this inequality directly, and by \eqref{heat 2}, we obtain the following decay result:
\begin{equation}\label{heat 555}
\norm{\nabla^\ell u(t)}_{L^{2}}^{2}\leq \left( \norm{\nabla^\ell u_0}%
_{L^{2}}^{-\frac{2}{\ell +s}}+\frac{C_{0}t}{\ell +s}\right) ^{-(\ell
+s)}\leq C_{0}(1+t)^{-(\ell +s)}.
\end{equation}
Hence, we conclude our decay results by the pure energy method. Although \eqref{heat 555} can be proved by the Fourier analysis or
spectral method, the same strategy in our proof can be applied to nonlinear
system with two essential points in the proof: (1) closing the energy
estimates at each $\ell $-th level (referring to the order of the spatial
derivatives of the solution); (2) deriving a novel negative Sobolev
estimates for nonlinear system which requires $s<3/2$ ($n/2$ for dimension $n$).

In the rest of this paper, except that we will collect in Appendix the analytic tools which will be used, we will apply the energy method illustrated above to prove Theorem \ref{1mainth}. However, we will be
not able to close the energy estimates at each $\ell $-th level as the
heat equation. This is caused by the ``degenerate" dissipative structure of the NSP system when using our energy method. More precisely, the linear energy identity of the problem reads as: for $k=0,\dots,N$,
\begin{equation}\label{1energy identity}
\frac{1}{2}\frac{d}{dt}\int_{\mathbb{R}^3}|\nabla^k \varrho|^2+|\nabla^k u|^2+|\nabla^k \nabla\Phi|^2\,dx
+\int_{\mathbb{R}^3} {\mu}|\nabla \nabla^k u|^2+( {\mu}+ {\lambda})|\diverge \nabla^k u|^2\,dx=0.
\end{equation}
The constraint \eqref{1viscosity} implies that there exists a constant $\sigma_0>0$ such that
\begin{equation}\label{1coercive}
\int_{\mathbb{R}^3} {\mu}|\nabla \nabla^k u|^2+( {\mu}+ {\lambda})|\diverge \nabla^k u|^2\,dx
\ge \sigma_0 \norm{\nabla^{k+1} u}_{L^2}^2.
\end{equation}
Note that \eqref{1energy identity} and \eqref{1coercive} only give the dissipative estimate for $u$. To rediscover the dissipative estimate for $\varrho$ and $\nabla\Phi$, we will use the equations  via constructing the interactive energy functional between $u$ and $\nabla\varrho$ to deduce
\begin{equation}\label{1daf}
\begin{split}
&\frac{d}{dt}\int_{\mathbb{R}^3}  \nabla^ku\cdot\nabla\nabla^k\varrho\,dx
+C\left(\norm{\nabla^k\varrho}_{L^2}^2+\norm{\nabla^{k+1}\varrho}_{L^2}^2+\norm{\nabla^{k+1}\nabla\Phi}_{L^2}^2 +\norm{\nabla^{k+2}\nabla\Phi}_{L^2}^2\right)
\\&\quad \lesssim\norm{\nabla^{k+1}u}_{L^2}^2+\norm{\nabla^{k+2}u}_{L^2}^2.
\end{split}
\end{equation}
This implies that to get the dissipative estimate for $\varrho$ and $\nabla\Phi$ it requires us to do the energy estimates \eqref{1energy identity} at both the $k$-th and the ($k+1$)-th levels (referring to the order of the spatial derivatives of the solution). To get around this obstacle, the idea is to construct some energy functional  ${\mathcal{E}}%
_{\ell }^{m}(t),\ 0\le\ell\le m-1$ with $1\le m\le N$ ($\ell$ less than $N-1$ is restricted by \eqref{1daf}),
\begin{equation}
{\mathcal{E}}_{\ell }^{m}(t)\backsim \sum_{\ell \leq k\leq m}%
\norm{\left[\nabla^k \varrho(t),\nabla^k u(t),\nabla^k \nabla\Phi(t)\right]}_{L^2}^{2},
\end{equation}%
which has a \textit{minimum} derivative count $\ell .$ We will then close the energy estimates at each $\ell $-th level in a weak sense by deriving
the Lyapunov-type inequality (cf. \eqref{1proof5}) for these energy
functionals in which the corresponding dissipation (denoted by ${\mathcal{D}}%
_{\ell }^{m}(t)$) can be related to the energy ${\mathcal{E}}_{\ell }^{m}(t)$
similarly as \eqref{heat 4} by the Sobolev interpolation. This can be easily
established for the linear homogeneous problem along our analysis, however,
for the nonlinear problem \eqref{1NS2} it is much more complicated due to
the nonlinear estimates. This is the second point of this paper that we will
extensively and carefully use the Sobolev interpolation of the
Gagliardo-Nirenberg inequality between high-order and low-order spatial
derivatives to bound the nonlinear terms by $\sqrt{\mathcal{E}_{0}^{3}(t)}{\mathcal{D}}_{\ell }^{m}(t)$ that can be absorbed. When deriving the negative Sobolev
estimates, we need to restrict that $s<3/2$ in order to estimate $\Lambda ^{-s}$ acting on the
nonlinear terms by using the Hardy-Littlewood-Sobolev inequality, and also
we need to separate the cases that $s\in (0,1/2]$ and $s\in (1/2,3/2)$. Once
these estimates are obtained, Theorem \ref{1mainth} follows by the
interpolation between negative and positive Sobolev norms similarly as the heat equation case.

To end this introduction, we will compare the NSP system \eqref{1NS} with some related models with the electric force. The mostly related model is that the compressible bipolar Navier-Stokes-Poisson system (BNSP) of describing the dynamic of charged particles of two carrier type (e.g., ions and electrons). It is observed in \cite{LYZ,HL} that the BNSP system can be reformulated to an equivalent system consisting of the NS system for the sum of densities and velocities, $(\rho, u)$, and the NSP system for the difference, $(d, w,\nabla\Phi)$, which are coupled with each other through the nonlinear terms. Then for the linearized BNSP system, $(\rho, u)$ decays as the NS system and $(d, w,\nabla\Phi)$ decays as the NSP system. However, for the nonlinear BNSP system there is a new difficulty arising when estimating the nonlinear interactive terms. Precisely, there is one term $\varrho  w \nabla\Phi$ that we can not bound it through our energy method since $L^2$ norm of these three functions are all not included in the dissipation rate. Hence, it is interesting to apply our energy method to the BNSP system. But in the study of the kinetic models of the Vlasov-Poisson-Boltzmann system, the situation is contrary. The  one-species Vlasov-Poisson-Boltzmann system \cite{DS1} will encounter the similar difficulty as the BNSP system when applying our energy method. But due to the special cancelation property between two species which gives the dissipative estimates of the $L^2$ norm of the electric field, in \cite{W} we have successfully applied our energy method to the two-species Vlasov-Poisson-Boltzmann system to show the decay rates of the solution. The decay result itself is very attractive: the total density of
two species of particles decays at the optimal algebraic rate as the Boltzmann equation, but the disparity between two species and the electric field decay at an exponential rate. Finally, it is also interesting to  compare with the compressible Euler-Poisson system (EP) without the viscosity. It was shown in \cite{G} that for the irrotational EP system both density and velocity decay in $L^\infty$ norm at the rate $(1+t)^{-\beta}$ for any $\beta\in (1,3/2)$. Compared to our result (cf. \eqref{Linfty decay}) it can be comprehended that the dissipative effect of the  viscosity enhances the decay rate of the density.

%%%%%%%%%%%%%%%%%%%%%%%%%%%%%%%%%%%%%%%%%%%%%%%
\section{Energy estimates}\label{energy estimate}
%%%%%%%%%%%%%%%%%%%%%%%%%%%%%%%%%%%%%%%%%%%%%%%

Denoting $\varrho=\rho-1$, we rewrite \eqref{1NS} in the perturbation form as
\begin{equation}\label{1NS2}
\left\{\begin{array}{lll} \displaystyle\partial_t\varrho +{\rm div}u=-{\rm div}(\varrho u)
\\\displaystyle\partial_tu - \mu \Delta u -( \mu + \lambda )\nabla{\rm div}u+ \nabla\varrho-\nabla\Phi
=-u\cdot\nabla u -h(\varrho)\left( \mu \Delta u+( \mu + \lambda )\nabla{\rm div} u\right)-f(\varrho)\nabla\varrho
\\\Delta\Phi=\varrho
\\ (\varrho,u)|_{t=0}=(\varrho_0,  u _0),
\end{array}\right.
\end{equation}
where the two nonlinear functions of $\varrho$ are defined by
\begin{equation}\label{1 h and f}
h(\varrho):=\frac{\varrho}{\varrho+1} \hbox{ and }
f(\varrho):=\frac{p'(\varrho+1)}{{\varrho+1}}-1.
\end{equation}
In this section, we will derive the a priori energy estimates for the equivalent system \eqref{1NS2}. Hence we assume a priori that for
sufficiently small $\delta>0$,
\begin{equation}\label{1a priori}
\sqrt{\mathcal{E}_0^3(t)}=\norm{\varrho(t)}_{H^{3}}+\norm{u(t)}_{H^{3}}+\norm{\nabla\Phi(t)}_{H^{3}}\le \delta.
\end{equation}

First of all, by \eqref{1a priori} and Sobolev's inequality, we obtain
\begin{equation}
{1}/{2}\le \varrho+1\le  2.
\end{equation}
Hence, we immediately have
\begin{equation} \label{1hf}
|h(\varrho)|,|f(\varrho)|\le  C|\varrho|\hbox{ and } |h^{(k)}(\varrho)|,|f^{(k)}(\varrho)| \le C\hbox{ for any }k\ge 1.
\end{equation}

We first derive the following energy estimates which contains the dissipation estimate for $u$.
\begin{lemma}\label{1Ekle}
If $\sqrt{\mathcal{E}_0^3(t)}\le \delta$, then for $k=0,\dots,N$, we have
\begin{equation}\label{1E_k}
\begin{split}
&\frac{d}{dt}\int_{\mathbb{R}^3}  |\nabla^{k} \varrho|^2+|\nabla^{k} u|^2+ |\nabla^k\nabla\Phi|^2\,dx+C\norm{\nabla^{k+1}
u}_{L^2}^2
\\&\quad\lesssim  \sqrt{\mathcal{E}_0^3}\left(\norm{\nabla^{k}\varrho}_{L^2}^2+\norm{\nabla^{k+1}u}_{L^2}^2+\norm{\nabla^{k+1}\nabla\Phi}_{L^2}^2\right).
\end{split}
\end{equation}
\end{lemma}

\begin{proof}
Applying $\nabla^{k}$ to $\eqref{1NS2}_1, \eqref{1NS2}_2$ and multiplying the resulting identities by $ \nabla^{k}\varrho$, $\nabla^{k} u$ respectively, summing up them and
then integrating over $\mathbb{R}^3$ by parts, we obtain
\begin{equation}\label{1E_k_0}
\begin{split}
&\frac{1}{2}\frac{d}{dt}\int_{\mathbb{R}^3}  |\nabla^{k} \varrho|^2+|\nabla^{k} u|^2\,dx +\int_{\mathbb{R}^3} \mu |\nabla^{k+1}
u|^2+( \mu + \lambda )|\nabla^{k}{\rm div}u|^2\,dx -\int_{\mathbb{R}^3}\nabla\nabla^k\Phi\cdot \nabla^k u\,dx
\\&\quad=\int_{\mathbb{R}^3} \nabla^{k} (-\varrho{\rm div}u-u\cdot\nabla\varrho)\nabla^{k}\varrho
-\nabla^{k}\left(u\cdot\nabla u+h(\varrho)( \mu \Delta u+( \mu + \lambda )\nabla{\rm div} u)+f(\varrho)\nabla\varrho\right)\cdot\nabla^{k} u\,dx
\\&\quad:=J_1+J_2+J_3+J_4+J_5.
\end{split}
\end{equation}

We shall first estimate each term in the right hand side of \eqref{1E_k_0}. The main idea is that we will carefully interpolate the spatial derivatives between the higher-order derivatives and the lower-order derivatives to bound these nonlinear terms by the right hand side of \eqref{1E_k}. First, for the term $J_1$, employing the Leibniz formula and by H\"older's
inequality, we obtain
\begin{equation}\label{1E_k_1_0}
\begin{split}
J_1&:=- \int_{\mathbb{R}^3}\nabla^{k} (\varrho{\rm div}u)\nabla^{k}\varrho\,dx
=-  \sum_{0\le\ell\le k}C_{k}^\ell\int_{\mathbb{R}^3}  \nabla^\ell\varrho \nabla^{k-\ell}{\rm div}u \nabla^{k}\varrho\,dx
\\&\lesssim  \sum_{0\le\ell\le k} \norm{\nabla^\ell\varrho \nabla^{k-\ell+1}u}_{L^2}\norm{\nabla^{k}\varrho}_{L^2}.
\end{split}
\end{equation}
To estimate the first factor in the above, we take the $L^\infty$-norm on the term with less number of derivatives. Hence, if $\ell\le \left[\frac{k}{2}\right]$, together with the Sobolev interpolation of Lemma \ref{1interpolation}, we have
\begin{equation}\label{1E_k_1_1}
\begin{split}
\norm{\nabla^\ell\varrho \nabla^{k-\ell+1}u}_{L^2}&\lesssim \norm{\nabla^\ell\varrho}_{L^\infty} \norm{\nabla^{k-\ell+1}u}_{L^2}
\\&\lesssim \norm{\nabla^\alpha\varrho}_{L^2} ^{1-\frac{\ell}{k}}\norm{\nabla^k\varrho}_{L^2}^\frac{\ell}{k} \norm{\nabla u}_{L^2}^\frac{\ell}{k}\norm{\nabla^{k+1}u}_{L^2}^{1-\frac{\ell}{k}}.
\end{split}
\end{equation}
Here $\alpha$ comes from the adjustment of the index between the energy and the dissipation, which is defined by
\begin{equation}\label{alpha}
\begin{split}
&\frac{\ell}{3}=\left(\frac{\alpha}{3}-\frac{1}{2}\right)\times\left(1-\frac{\ell}{k}\right)
+\left(\frac{k}{3}-\frac{1}{2}\right)\times \frac{\ell}{k}
\\ &\quad \Longrightarrow \alpha=\frac{ 3k}{2(k-\ell)}\in \left[\frac{3}{2},3\right] \, \text{ since }\ell\le \frac{k}{2}.
\end{split}
\end{equation}
Hence, by the definition of the energy $\mathcal{E}_0^3$ and Young's inequality, we obtain that for $\ell\le \left[\frac{k}{2}\right]$,
\begin{equation}\label{1E_k_1_1`}
\norm{\nabla^\ell\varrho \nabla^{k-\ell+1}u}_{L^2}\lesssim\sqrt{\mathcal{E}_0^3}\left(\norm{\nabla^{k}\varrho}_{L^2} +\norm{\nabla^{k+1}u}_{L^2} \right).
\end{equation}
If $\left[\frac{k}{2}\right]+1\le \ell\le k $ (if $k<\left[\frac{k}{2}\right]+1$, then it's nothing in this case, and hereafter, etc.), we have
\begin{equation}\label{1E_k_1_2}
\begin{split}
\norm{\nabla^\ell\varrho \nabla^{k-\ell+1}u}_{L^2}&\lesssim \norm{\nabla^\ell\varrho}_{L^2} \norm{\nabla^{k-\ell+1}u}_{L^\infty}
\\&\lesssim \norm{\nabla \varrho}_{L^2}^{1-\frac{\ell-1}{k-1}}  \norm{\nabla^k \varrho}_{L^2}^{\frac{\ell-1}{k-1}} \norm{\nabla^\alpha u}_{L^2}^{\frac{\ell-1}{k-1}}\norm{\nabla^{k+1}u}_{L^2}^{1-\frac{\ell-1}{k-1}}
\\&\lesssim\sqrt{\mathcal{E}_0^3}\left(\norm{\nabla^{k}\varrho}_{L^2} +\norm{\nabla^{k+1}u}_{L^2} \right),
\end{split}
\end{equation}
where $\alpha$ is defined by
\begin{equation}
\begin{split}
&\frac{k+1-\ell}{3}=\left(\frac{\alpha}{3}-\frac{1}{2}\right)\times\frac{\ell-1}{k-1}
+\left(\frac{k+1}{3}-\frac{1}{2}\right)\times\left(1-\frac{\ell-1}{k-1}\right)
\\ &\quad \Longrightarrow \alpha=\frac{ k-1}{ 2(\ell-1) }+2\in \left[\frac{5}{2},3\right]\, \text{ since }\ell\ge \frac{k+1}{2}.
\end{split}
\end{equation}
In light of \eqref{1E_k_1_1`} and \eqref{1E_k_1_2},  by Cauchy's inequality, we deduce from \eqref{1E_k_1_0} that
\begin{equation}\label{1E_k_1}
J_{1}\lesssim  \sqrt{\mathcal{E}_0^3}\left(\norm{\nabla^{k}\varrho}_{L^2}^2+\norm{\nabla^{k+1}u}_{L^2}^2\right).
\end{equation}

Next, for the term $J_2$, we utilize the commutator notation \eqref{1commuta} to rewrite it as
\begin{equation}\label{1E_k_2_0}
\begin{split}
J_2&:=- \int_{\mathbb{R}^3}\nabla^{k} (u\cdot\nabla\varrho)\nabla^{k}\varrho\,dx
=- \int_{\mathbb{R}^3} \left(u\cdot\nabla \nabla^{k}\varrho+[\nabla^{k},u]\cdot\nabla\varrho\right)\nabla^{k}\varrho\,dx
\\&:=J_{21}+J_{22}.
\end{split}
\end{equation}
By integrating by part, by Sobolev's inequality, we have
\begin{equation}\label{1E_k_2_1}
\begin{split}
J_{21}&=- \int_{\mathbb{R}^3}  u\cdot\nabla \frac{|\nabla^{k}\varrho|^2 }{2}\,dx =\frac{1 }{2}\int_{\mathbb{R}^3}  {\rm div}u\,
{|\nabla^{k}\varrho|^2 }\,dx
\\&\lesssim\norm{\nabla u}_{L^\infty}\norm{\nabla^{k}\varrho}_{L^2}^2\lesssim \sqrt{\mathcal{E}_0^3} \norm{\nabla^{k}\varrho}_{L^2}^2.
\end{split}
\end{equation}
We use the commutator estimate of Lemma \ref{1commutator} and Sobolev's inequality to bound
\begin{equation}\label{1E_k_2_2}
\begin{split}
J_{22} &\lesssim \left(\norm{\nabla u}_{L^\infty}\norm{\nabla^{k-1}\nabla\varrho}_{L^2}+\norm{\nabla^{k}u }_{L^6}\norm{\nabla
\varrho}_{L^3}\right)\norm{\nabla^{k}\varrho}_{L^2}
\\&\lesssim \sqrt{\mathcal{E}_0^3} \left(\norm{\nabla^{k}\varrho}_{L^2}^2+\norm{\nabla^{k+1}u}_{L^2}^2\right).
\end{split}
\end{equation}
In light of \eqref{1E_k_2_1}--\eqref{1E_k_2_2}, we find
\begin{equation}\label{1E_k_2}
J_{2}\lesssim \sqrt{\mathcal{E}_0^3}\left(\norm{\nabla^{k}\varrho}_{L^2}^2+\norm{\nabla^{k+1}u}_{L^2}^2\right).
\end{equation}

We now estimate the term $J_3$. By H\"older's  and Sobolev's inequalities, we obtain
\begin{equation}\label{E_k_3_0}
\begin{split}
J_3&=-\int_{\mathbb{R}^3} \nabla^{k}\left(u\cdot\nabla u \right)\cdot\nabla^{k} u\,dx
=-\sum_{0\le \ell\le k}C_k^\ell\int_{\mathbb{R}^3}  \left(\nabla^{\ell}u\cdot\nabla\nabla^{k-\ell} u \right) \cdot\nabla^{k} u\,dx
\\&\lesssim \sum_{0\le \ell\le k}\norm{\nabla^{\ell} u\cdot \nabla^{k-\ell+1} u }_{L^{\frac{6}{5}}}\norm{\nabla^{k} u}_{L^6}
\\&\lesssim \sum_{0\le \ell\le k}\norm{\nabla^{\ell} u\cdot \nabla^{k-\ell+1} u }_{L^{\frac{6}{5}}}\norm{\nabla^{k+1} u}_{L^2}.
\end{split}
\end{equation}
If $\ell\le \left[\frac{k}{2}\right]$, by H\"older's inequality and Lemma \ref{1interpolation}, we have
\begin{equation}\label{1E_k_1_9}
\begin{split}
\norm{\nabla^{\ell} u\cdot \nabla^{k-\ell+1} u }_{L^{\frac{6}{5}}}&\lesssim\norm{\nabla^{\ell} u}_{L^3}\norm{ \nabla^{k-\ell+1} u }_{L^2}
\\&\lesssim \norm{\nabla^\alpha u}_{L^2} ^{1-\frac{\ell}{k+1}}\norm{\nabla^{k+1}u}_{L^2}^\frac{\ell}{k+1} \norm{ u}_{L^2}^\frac{\ell}{k+1}\norm{\nabla^{k+1}u}_{L^2}^{1-\frac{\ell}{k+1}}
\\&\lesssim  \sqrt{\mathcal{E}_0^3}\norm{\nabla^{k+1} u}_{L^2},
\end{split}
\end{equation}
where $\alpha$ is defined by
\begin{equation}\label{aa2}
\begin{split}
&\frac{\ell}{3}-\frac{1}{3}=\left(\frac{\alpha}{3}-\frac{1}{2}\right)\times\left(1-\frac{\ell}{k+1}\right)
+\left(\frac{k+1}{3}-\frac{1}{2}\right)\times \frac{\ell}{k+1}
\\ &\quad \Longrightarrow \alpha=\frac{ k+1 }{2(k+1-\ell) }\in \left[\frac{1}{2},1\right) \, \text{ since }\ell\le \frac{k}{2}.
\end{split}
\end{equation}
If $\ell\ge \left[\frac{k}{2}\right]+1$, by H\"older's inequality and Lemma \ref{1interpolation} again, we have
\begin{equation}\label{1E_k_1_000}
\begin{split}
\norm{\nabla^{\ell} u\cdot \nabla^{k-\ell+1} u }_{L^{\frac{6}{5}}}&\lesssim\norm{\nabla^{\ell} u}_{L^2}\norm{ \nabla^{k-\ell+1} u }_{L^3}
\\&\lesssim \norm{  u}_{L^2} ^{1-\frac{\ell}{k+1}}\norm{\nabla^{k+1}u}_{L^2}^\frac{\ell}{k+1} \norm{ \nabla^\alpha u}_{L^2}^\frac{\ell}{k+1}\norm{\nabla^{k+1}u}_{L^2}^{1-\frac{\ell}{k+1}}
\\&\lesssim  \sqrt{\mathcal{E}_0^3}\norm{\nabla^{k+1} u}_{L^2},
\end{split}
\end{equation}
where $\alpha$ is defined by
\begin{equation}
\begin{split}
 &\frac{k-\ell+1}{3}-\frac{1}{3}=\left(\frac{\alpha}{3}-\frac{1}{2}\right)\times\frac{\ell}{k+1}
+\left(\frac{k+1}{3}-\frac{1}{2}\right)\times\left(1-\frac{\ell}{k+1}\right)
\\&\quad\Longrightarrow \alpha=\frac{ k+1 }{2\ell } \in \left(\frac{1}{2},1\right]\, \text{ since }\ell\ge \frac{k+1}{2}.
\end{split}
\end{equation}
In light of \eqref{1E_k_1_9} and \eqref{1E_k_1_000},  we deduce from \eqref{E_k_3_0} that
\begin{equation}\label{1E_k_3}
J_{3} \lesssim  \sqrt{\mathcal{E}_0^3}\norm{\nabla^{k+1} u}_{L^2}^2.
\end{equation}

Next, we estimate the term $J_4$. We do the approximation to simplify the presentations by
\begin{equation}\label{1E_k_4_0}
\begin{split}
J_4&:=-\int_{\mathbb{R}^3} \nabla^{k}\left( h(\varrho)( \mu \Delta u+( \mu + \lambda )\nabla{\rm div} u)\right)\cdot\nabla^{k} u\,dx
\approx -\int_{\mathbb{R}^3} \nabla^{k} \left(h(\varrho)\nabla^2u\right)\cdot\nabla^{k} u\,dx.
\end{split}
\end{equation}
If $k=0$, by the fact \eqref{1hf} and H\"older's and Cauchy's inequalities, we obtain
\begin{equation}\label{1E_k_4_11}
\begin{split}
J_4\approx-\int_{\mathbb{R}^3} h(\varrho)\nabla^2u\cdot u\,dx\lesssim \norm{\varrho}_{L^2}\norm{\nabla^2u}_{L^3}\norm{u}_{L^6}\lesssim \sqrt{\mathcal{E}_0^3}\left(\norm{\varrho}_{L^2}^2+\norm{\nabla u}_{L^2}^2\right).
\end{split}
\end{equation}
For $k\ge 1$, we integrate by parts to have
\begin{equation}
\begin{split}
J_4& \approx \int_{\mathbb{R}^3} \nabla^{k-1} \left(h(\varrho)\nabla^2u\right)\cdot\nabla^{k+1} u\,dx
=\sum_{0\le \ell\le k-1}C_{k-1}^\ell\int_{\mathbb{R}^3}\nabla^\ell h(\varrho)\nabla^{k-\ell+1}u  \cdot\nabla^{k+1} u \,dx
\\&\lesssim \sum_{0\le \ell\le k-1}\norm{\nabla^\ell h(\varrho)\nabla^{k-\ell+1}u}_{L^2}\norm{\nabla^{k+1} u}_{L^2}.
\end{split}
\end{equation}
If $0\le \ell\le \left[\frac{k}{2}\right]$, by \eqref{1hf}, Lemma \ref{infty} and the estimates in \eqref{1E_k_1_1}, we obtain
\begin{equation}\label{1E_k_4_13}
\begin{split}
 \norm{\nabla^\ell h(\varrho)\nabla^{k-\ell+1}u}_{L^2}&\lesssim \norm{\nabla^\ell h(\varrho)}_{L^\infty}\norm{\nabla^{k-\ell+1}u}_{L^2}
 \lesssim \norm{\nabla^\ell \varrho}_{L^\infty}\norm{\nabla^{k-\ell+1}u}_{L^2}
\\&
 \lesssim  \sqrt{\mathcal{E}_0^3}\left(\norm{\nabla^{k}\varrho}_{L^2}+\norm{\nabla^{k+1}u}_{L^2}\right).
\end{split}
\end{equation}
If $\left[\frac{k}{2}\right]+1\le \ell\le k-1$, we rewrite this factor to have
\begin{equation}\label{1E_k_4_14}
\begin{split}
 \norm{\nabla^\ell h(\varrho)\nabla^{k-\ell+1}u}_{L^2}&=\norm{\nabla^{\ell-1}\left( h'(\varrho)\nabla \varrho\right)\nabla^{k-\ell+1}u}_{L^2}
 =\norm{\sum_{0\le m\le \ell-1} C_{\ell-1}^m \nabla^m h'(\varrho)\nabla^{\ell-m} \varrho\nabla^{k-\ell+1}u}_{L^2}
\\&\lesssim\sum_{0\le m\le \ell-1}\norm{\nabla^m h'(\varrho)\nabla^{\ell-m}\varrho\nabla^{k-\ell+1}u}_{L^2}.
\end{split}
\end{equation}
For $m=0$, by the fact \eqref{1hf} and the estimates in \eqref{1E_k_1_2}, we have
\begin{equation}\label{1E_k_4_15}
  \norm{ h'(\varrho)\nabla^{\ell}\varrho\nabla^{k-\ell+1}u}_{L^2}\lesssim \norm{\nabla^{\ell}\varrho}_{L^2}\norm{\nabla^{k-\ell+1}u}_{L^\infty} \lesssim  \sqrt{\mathcal{E}_0^3}\left(\norm{\nabla^{k}\varrho}_{L^2}+\norm{\nabla^{k+1}u}_{L^2}\right).
\end{equation}
For $1\le m\le \ell-1$, by Lemma \ref{infty} and Lemma \ref{1interpolation}, we have
\begin{equation}\label{1E_k_4_16}
\begin{split}
  &\norm{\nabla^m h'(\varrho)\nabla^{\ell-m}\varrho\nabla^{k-\ell+1}u}_{L^2}
\\&\quad \lesssim  \norm{\nabla^m h'(\varrho)}_{L^\infty}\norm{\nabla^{\ell-m}\varrho}_{L^\infty}\norm{\nabla^{k-\ell+1}u}_{L^2}
  \lesssim  \norm{\nabla^m \varrho}_{L^\infty}\norm{\nabla^{\ell-m}\varrho}_{L^\infty}\norm{\nabla^{k-\ell+1}u}_{L^2}
\\&\quad \lesssim  \norm{\nabla^2 \varrho}_{L^2}^{1-\frac{m-\frac{1}{2}}{k-2}}\norm{\nabla^k \varrho}_{L^2}^{\frac{m-\frac{1}{2}}{k-2}}
\norm{\nabla^2 \varrho}_{L^2}^{1-\frac{\ell-m-\frac{1}{2}}{k-2}}\norm{\nabla^k \varrho}_{L^2}^{\frac{\ell-m-\frac{1}{2}}{k-2}} \norm{\nabla^{k-\ell+1}u}_{L^2}
\\&\quad \lesssim  \norm{\nabla^2 \varrho}_{L^2}^{2-\frac{\ell-1}{k-2}}\norm{\nabla^k \varrho}_{L^2}^{\frac{\ell-1}{k-2}}
\norm{\nabla^\alpha u}_{L^2}^{\frac{\ell-1}{k-2}}\norm{\nabla^{k+1}u}_{L^2}^{1-\frac{\ell-1}{k-2}}
\\&\quad\lesssim   {\mathcal{E}_0^3}\left(\norm{\nabla^{k}\varrho}_{L^2} +\norm{\nabla^{k+1}u}_{L^2} \right),
\end{split}
\end{equation}
where $\alpha$ is defined by
\begin{equation}
\begin{split}
 &k-\ell+1=\alpha\times \frac{\ell-1}{k-2}
+ (k+1)\times\left(1-\frac{\ell-1}{k-2}\right)
\\&\quad\Longrightarrow \alpha=3-\frac{k-2}{\ell-1}\in (1,2]\, \text{ since } \frac{k+1}{2}\le \ell\le k-1.
\end{split}
\end{equation}
In light of \eqref{1E_k_4_15} and \eqref{1E_k_4_16}, we deduce from \eqref{1E_k_4_14} that for $\left[\frac{k}{2}\right]+1\le \ell\le k-1$,
\begin{equation} \label{1E_k_4_17}
\norm{\nabla^\ell h(\varrho)\nabla^{k-\ell+1}u}_{L^2}\lesssim  \sqrt{\mathcal{E}_0^3}\left(\norm{\nabla^{k}\varrho}_{L^2}+\norm{\nabla^{k+1}u}_{L^2}\right).
\end{equation}
This together with \eqref{1E_k_4_13} and \eqref{1E_k_4_11} implies that
\begin{equation} \label{1E_k_4}
J_4\lesssim  \sqrt{\mathcal{E}_0^3}\left(\norm{\nabla^{k}\varrho}_{L^2}^2+\norm{\nabla^{k+1}u}_{L^2}^2\right).
\end{equation}

Finally, it remains to estimate the last term $J_5$.  If $k=0$, by the fact \eqref{1hf} and H\"older's and Cauchy's inequalities, we obtain
\begin{equation}\label{1E_k_5_11}
J_5=-\int_{\mathbb{R}^3} f(\varrho)\nabla\varrho\cdot u\,dx\lesssim \norm{\varrho}_{L^2}\norm{\nabla\varrho}_{L^3}\norm{u}_{L^6}\lesssim \sqrt{\mathcal{E}_0^3}\left(\norm{\varrho}_{L^2}^2+\norm{\nabla u}_{L^2}^2\right).
\end{equation}
For $k\ge 1$, we integrate by parts to have
\begin{equation}\label{1E_k_5_0}
\begin{split}
J_5&=\int_{\mathbb{R}^3} \nabla^{k-1}\left( f(\varrho)\nabla\varrho\right)\cdot\nabla^{k+1} u\,dx
=\sum_{0\le \ell\le k-1}C_{k-1}^\ell\int_{\mathbb{R}^3}\nabla^\ell f(\varrho)\nabla^{k-\ell}\varrho  \cdot\nabla^{k+1} u \,dx
\\&\lesssim \sum_{0\le \ell\le k-1}\norm{\nabla^\ell f(\varrho)\nabla^{k-\ell}\varrho}_{L^2}\norm{\nabla^{k+1} u}_{L^2}.
\end{split}
\end{equation}
If $0\le \ell\le \left[\frac{k}{2}\right]$, by Lemma \ref{infty} and Lemma \ref{1interpolation}, we have
\begin{equation}\label{1E_k_5_13}
\begin{split}
 \norm{\nabla^\ell f(\varrho)\nabla^{k-\ell}\varrho}_{L^2}&\lesssim \norm{\nabla^\ell f(\varrho)}_{L^\infty}\norm{\nabla^{k-\ell}\varrho}_{L^2}
 \lesssim \norm{\nabla^\ell \varrho}_{L^\infty}\norm{\nabla^{k-\ell}\varrho}_{L^2}
\\& \lesssim \norm{\nabla^\alpha\varrho}_{L^2} ^{1-\frac{\ell}{k}}\norm{\nabla^k\varrho}_{L^2}^\frac{\ell}{k} \norm{\varrho}_{L^2}^\frac{\ell}{k}\norm{\nabla^{k}\varrho}_{L^2}^{1-\frac{\ell}{k}}\lesssim \sqrt{\mathcal{E}_0^3}\norm{\nabla^k\varrho}_{L^2},
\end{split}
\end{equation}
where $\alpha$ is the same one defined by \eqref{alpha}. If $\left[\frac{k}{2}\right]+1\le \ell\le k-1$, we rewrite this factor to have
\begin{equation}\label{1E_k_5_14}
\begin{split}
 \norm{\nabla^\ell f(\varrho)\nabla^{k-\ell}\varrho}_{L^2}&=\norm{\nabla^{\ell-1}\left( f'(\varrho)\nabla \varrho\right)\nabla^{k-\ell}\varrho}_{L^2}
 =\norm{\sum_{0\le m\le \ell-1} C_{\ell-1}^m \nabla^m f'(\varrho)\nabla^{\ell-m} \varrho\nabla^{k-\ell }\varrho}_{L^2}
\\&\lesssim\sum_{0\le m\le \ell-1}\norm{\nabla^m f'(\varrho)\nabla^{\ell-m}\varrho\nabla^{k-\ell}\varrho}_{L^2}.
\end{split}
\end{equation}
For $m=0$, by the fact  \eqref{1hf} and Lemma \ref{1interpolation}, we have
\begin{equation}\label{1E_k_5_15}
\begin{split}
  \norm{ f'(\varrho)\nabla^{\ell}\varrho\nabla^{k-\ell}\varrho}_{L^2}&\lesssim \norm{\nabla^{\ell}\varrho}_{L^2}\norm{\nabla^{k-\ell}\varrho}_{L^\infty}
  \\&\lesssim \norm{\varrho}_{L^2}^{1-\frac{\ell}{k}}  \norm{\nabla^k \varrho}_{L^2}^{\frac{\ell}{k}} \norm{\nabla^\alpha \varrho}_{L^2}^{\frac{\ell}{k}}\norm{\nabla^{k}\varrho}_{L^2}^{1-\frac{\ell}{k}}\lesssim \sqrt{\mathcal{E}_0^3}\norm{\nabla^k\varrho}_{L^2},
\end{split}
\end{equation}
where $\alpha$ is defined by
\begin{equation}
\begin{split}
&\frac{k-\ell}{3}=\left(\frac{\alpha}{3}-\frac{1}{2}\right)\times\frac{\ell}{k}
+\left(\frac{k}{3}-\frac{1}{2}\right)\times\left(1-\frac{\ell}{k}\right)
\\ &\quad \Longrightarrow \alpha=\frac{ 3k}{ 2\ell }\le 3\, \text{ since }  \ell\ge\frac{k+1}{2}.
\end{split}
\end{equation}
For $1\le m\le \ell-1$, by Lemma \ref{infty} and Lemma \ref{1interpolation}, we have
\begin{equation}\label{1E_k_5_16}
\begin{split}
&  \norm{\nabla^m f'(\varrho)\nabla^{\ell-m}\varrho\nabla^{k-\ell}\varrho}_{L^2}
\\&\quad \lesssim  \norm{\nabla^m f'(\varrho)}_{L^\infty}\norm{\nabla^{\ell-m}\varrho}_{L^\infty}\norm{\nabla^{k-\ell}\varrho}_{L^2}
 \lesssim  \norm{\nabla^m \varrho}_{L^\infty}\norm{\nabla^{\ell-m}\varrho}_{L^\infty}\norm{\nabla^{k-\ell}\varrho}_{L^2}
\\&\quad \lesssim  \norm{\nabla^2 \varrho}_{L^2}^{1-\frac{m-\frac{1}{2}}{k-2}}\norm{\nabla^k \varrho}_{L^2}^{\frac{m-\frac{1}{2}}{k-2}}
\norm{\nabla^2 \varrho}_{L^2}^{1-\frac{\ell-m-\frac{1}{2}}{k-2}}\norm{\nabla^k \varrho}_{L^2}^{\frac{\ell-m-\frac{1}{2}}{k-2}} \norm{\nabla^{k-\ell}\varrho}_{L^2}
\\&\quad \lesssim  \norm{\nabla^2 \varrho}_{L^2}^{2-\frac{\ell-1}{k-2}}\norm{\nabla^k \varrho}_{L^2}^{\frac{\ell-1}{k-2}}
\norm{\nabla^\alpha \varrho}_{L^2}^{\frac{\ell-1}{k-2}}\norm{\nabla^{k}\varrho}_{L^2}^{1-\frac{\ell-1}{k-2}}\\&\quad \lesssim  {\mathcal{E}_0^3} \norm{\nabla^{k}\varrho}_{L^2},
\end{split}
\end{equation}
where $\alpha$ is defined by
\begin{equation}
\begin{split}
 &k-\ell=\alpha\times \frac{\ell-1}{k-2}
+ k\times\left(1-\frac{\ell-1}{k-2}\right)
\\&\quad\Longrightarrow \alpha=2-\frac{k-2}{\ell-1}\in (0,1]\, \text{ since } \frac{k+1}{2}\le \ell\le k-1.
\end{split}
\end{equation}
In light of \eqref{1E_k_5_15} and \eqref{1E_k_5_16}, we deduce from \eqref{1E_k_5_14} that for $\left[\frac{k}{2}\right]+1\le \ell\le k-1$,
\begin{equation} \label{1E_k_5_17}
\norm{\nabla^\ell f(\varrho)\nabla^{k-\ell}\varrho}_{L^2}\lesssim  \sqrt{\mathcal{E}_0^3} \norm{\nabla^{k}\varrho}_{L^2}.
\end{equation}
This together with \eqref{1E_k_5_13} and \eqref{1E_k_5_11} implies that
\begin{equation}\label{1E_k_5}
J_{5}\lesssim  \sqrt{\mathcal{E}_0^3}\left(\norm{\nabla^{k}\varrho}_{L^2}^2+\norm{\nabla^{k+1}u}_{L^2}^2\right).
\end{equation}

Now we turn to estimate the left hand side of \eqref{1E_k_0}. For the second term, we have
\begin{equation}\label{1E_k_66}
\int_{\mathbb{R}^3} \mu |\nabla^{k+1}
u|^2+( \mu + \lambda )|\nabla^{k}{\rm div}u|^2\,dx\ge \sigma_0\norm{\nabla^{k+1}u}_{L^2}^2.
\end{equation}
While for the last term, by the continuity equation $\eqref{1NS2}_1$ and the Poisson equation
$\eqref{1NS2}_3$ and the integration by parts, we get
\begin{equation}\label{1E_k_poi_0}
\begin{split}
&-\int_{\mathbb{R}^3}\nabla\nabla^k\Phi\cdot \nabla^k u\,dx=\int_{\mathbb{R}^3}\nabla^k\Phi\,\nabla^k{\rm div }u\,dx
\\&\quad=- \int_{\mathbb{R}^3}\nabla^k\Phi\nabla^k\partial_t\varrho+ \nabla^k\Phi\nabla^k{\rm div }(\varrho u)\,dx
\\&\quad=- \int_{\mathbb{R}^3}\nabla^k\Phi\nabla^k\partial_t\Delta\Phi-\nabla^k(\varrho u)\cdot\nabla^k\nabla\Phi\,dx
\\&\quad= \frac{1}{2}\frac{d}{dt}\int_{\mathbb{R}^3}|\nabla^k\nabla\Phi|^2\,dx+ \int_{\mathbb{R}^3}\nabla^k(\varrho
u)\cdot\nabla^k\nabla\Phi\,dx.
\end{split}
\end{equation}
Notice carefully that we can not estimate the last term in \eqref{1E_k_poi_0} directly. For instance, we may fail to
bound $\int_{\mathbb{R}^3}\varrho \nabla^k u \cdot\nabla^k\nabla\Phi\,dx$ by the right hand side of \eqref{1E_k}. To overcome this obstacle, the key
point is to make full use of again the Poisson equation $\eqref{1NS2}_3$ to rewrite $\varrho=\Delta\Phi$. This idea was also used in \cite{TW}. Indeed, using $\eqref{1NS2}_3$ and the integration by parts, by H\"older's inequality and Lemma
\ref{1interpolation}, we obtain
\begin{equation}\label{1E_k_poi_1}
\begin{split}
& \int_{\mathbb{R}^3}\nabla^k(\varrho u)\cdot\nabla^k\nabla\Phi\,dx
= \int_{\mathbb{R}^3}\nabla^k(\Delta\Phi
u)\cdot\nabla^k\nabla\Phi\,dx
\\& \quad= \int_{\mathbb{R}^3}\nabla^k(\nabla\Phi\cdot\nabla
u)\cdot\nabla^k\nabla\Phi+\nabla^k(\nabla\Phi\cdot u)\cdot\nabla\nabla^k\nabla\Phi\,dx
\\&\quad= \int_{\mathbb{R}^3}  \sum_{0\le \ell\le k}C_k^\ell \nabla^\ell\nabla\Phi \cdot(\nabla^{k-\ell+1}u\cdot\nabla^k\nabla\Phi+\nabla^{k-\ell}u\cdot\nabla\nabla^k\nabla\Phi)\,dx
\\&\quad\lesssim \sum_{0\le \ell\le k}\norm{\nabla^ \ell \nabla\Phi}_{L^3}\left(\norm{ \nabla^{k-\ell+1} u }_{L^2}\norm{\nabla^{k} \nabla\Phi}_{L^6}+\norm{ \nabla^{k-\ell} u }_{L^6}\norm{\nabla^{k+1}
\nabla\Phi}_{L^2}\right)
\\&\quad\lesssim \sum_{0\le \ell\le k}\norm{\nabla^ \ell \nabla\Phi}_{L^3}\norm{ \nabla^{k-\ell+1} u }_{L^2}\norm{\nabla^{k+1} \nabla\Phi}_{L^2}.
\end{split}
\end{equation}
If $\ell\le \left[\frac{k}{2}\right]$, by  Lemma \ref{1interpolation}, we have
\begin{equation}\label{1E_k_1_9991}
\begin{split}
\norm{\nabla^ \ell \nabla\Phi}_{L^3}\norm{ \nabla^{k-\ell+1} u }_{L^2} &\lesssim \norm{\nabla^\alpha \nabla\Phi}_{L^2} ^{1-\frac{\ell}{k+1}}\norm{\nabla^{k+1}\nabla\Phi}_{L^2}^\frac{\ell}{k+1} \norm{ u}_{L^2}^\frac{\ell}{k+1}\norm{\nabla^{k+1}u}_{L^2}^{1-\frac{\ell}{k+1}}
\\&\lesssim  \sqrt{\mathcal{E}_0^3}\left(\norm{\nabla^{k+1}\nabla\Phi}_{L^2}+\norm{\nabla^{k+1}u}_{L^2}\right),
\end{split}
\end{equation}
where $\alpha$ is the same one defined by \eqref{aa2}. Now if $\ell\ge \left[\frac{k}{2}\right]+1$, by Lemma \ref{1interpolation} again, we have
\begin{equation}\label{1E_k_1_9992}
\begin{split}
\norm{\nabla^ \ell \nabla\Phi}_{L^3}\norm{ \nabla^{k-\ell+1} u }_{L^2} &\lesssim \norm{ \nabla\Phi}_{L^2} ^{1-\frac{\ell+\frac{1}{2}}{k+1}}\norm{\nabla^{k+1}\nabla\Phi}_{L^2}^\frac{\ell+\frac{1}{2}}{k+1} \norm{ \nabla^\alpha u}_{L^2}^\frac{\ell+\frac{1}{2}}{k+1}\norm{\nabla^{k+1}u}_{L^2}^{1-\frac{\ell+\frac{1}{2}}{k+1}}
\\&\lesssim  \sqrt{\mathcal{E}_0^3}\left(\norm{\nabla^{k+1}\nabla\Phi}_{L^2}+\norm{\nabla^{k+1}u}_{L^2}\right),
\end{split}
\end{equation}
where $\alpha$ is defined by
\begin{equation}
\begin{split}
 & k-\ell +1= \alpha\times \frac{\ell+\frac{1}{2}}{k+1}+(k+1)\times \frac{\ell+\frac{1}{2}}{k+1}
\\&\quad\Longrightarrow \alpha=\frac{ k+1 }{2 \ell+1  } \in \left(\frac{1}{2},1\right)\hbox{ since }\ell\ge \frac{k+1}{2}.
\end{split}
\end{equation}
In light of the estimates \eqref{1E_k_1_9991}--\eqref{1E_k_1_9992}, we deduce from \eqref{1E_k_poi_0} that
\begin{equation}\label{1E_k_poi}
-\int_{\mathbb{R}^3}\nabla\nabla^k\Phi\cdot \nabla^k u\,dx\ge \frac{1}{2}\frac{d}{dt}\int_{\mathbb{R}^3}|\nabla^k\nabla\Phi|^2\,dx- C\sqrt{\mathcal{E}_0^3}
\left(\norm{\nabla^{k+1}\nabla\Phi}_{L^2}^2+\norm{\nabla^{k+1}u}_{L^2}^2\right).
\end{equation}

Plugging the estimates for $J_1\sim J_5$, $i.e.$, \eqref{1E_k_1}, \eqref{1E_k_2}, \eqref{1E_k_3}, \eqref{1E_k_4} and \eqref{1E_k_5}, and the estimates
\eqref{1E_k_poi} and \eqref{1E_k_66} into \eqref{1E_k_0},  we get \eqref{1E_k}.
\end{proof}

The following lemma provides the dissipation estimate for $\varrho$ and $\nabla\Phi$.
\begin{lemma}\label{1EkEk+1le}
If $\sqrt{\mathcal{E}_0^3(t)}\le \delta$, then for $k=0,\dots,N-1$, we have
\begin{equation}\label{1E_kE_k+1}
\begin{split}
&\frac{d}{dt}\int_{\mathbb{R}^3}  \nabla^ku\cdot\nabla\nabla^k\varrho\,dx
+C\left(\norm{\nabla^k\varrho}_{L^2}^2+\norm{\nabla^{k+1}\varrho}_{L^2}^2+\norm{\nabla^{k+1}\nabla\Phi}_{L^2}^2 +\norm{\nabla^{k+2}\nabla\Phi}_{L^2}^2\right)
\\&\quad \lesssim\norm{\nabla^{k+1}u}_{L^2}^2+\norm{\nabla^{k+2}u}_{L^2}^2.
\end{split}
\end{equation}
\end{lemma}

\begin{proof}
Applying $\nabla^k$ to $\eqref{1NS2}_2$ and then taking the $L^2$ inner product with $\nabla\nabla^k\varrho$, we obtain
\begin{equation}\label{1E_kE_k+1_0}
\begin{split}
&  \int_{\mathbb{R}^3}  |\nabla\nabla^{k}\varrho|^2\,dx-\int_{\mathbb{R}^3}  \nabla^k\nabla\Phi\cdot \nabla\nabla^k\varrho\,dx
\\&\le
-\int_{\mathbb{R}^3} \nabla^k
\partial_tu \cdot\nabla\nabla^k\varrho\,dx+C\norm{\nabla^{k+2} u}_{L^2} \norm{\nabla^{k+1}
\varrho}_{L^2}
\\&\quad+\norm{\nabla^{k}\left(u\cdot\nabla
u+h(\varrho)( \mu \Delta u+( \mu + \lambda )\nabla{\rm div}
u)+f(\varrho)\nabla\varrho\right)}_{L^2}\norm{\nabla^{k+1} \varrho}_{L^2}.
\end{split}
\end{equation}

The delicate first term in the right hand side of \eqref{1E_kE_k+1_0} involves the time derivative, and the key idea is to integrate by parts in the
$t$-variable and use the continuity equation. Thus by $\eqref{1NS2}_1$ and integrating by parts  for both the $t$- and $x$-variables, we may estimate
\begin{equation}\label{1E_kE_k+1_1}
\begin{split}
&-\int_{\mathbb{R}^3} \nabla^{k} u_t\cdot\nabla\nabla^k\varrho\,dx
\\&\quad=-\frac{d}{dt}\int_{\mathbb{R}^3}   \nabla^{k}u\cdot\nabla\nabla^k\varrho\,dx-\int_{\mathbb{R}^3} \nabla^{k} {\rm div}u\cdot\nabla^{k}\varrho_t\,dx
\\&\quad=-\frac{d}{dt}\int_{\mathbb{R}^3} \nabla^{k}u\cdot\nabla\nabla^k\varrho\,dx+\norm{\nabla^{k} {\rm div}u}_{L^2}^2+\int_{\mathbb{R}^3}\nabla^{k} {\rm div}u\cdot\nabla^{k}{\rm div}(\varrho u)\,dx.
\end{split}
\end{equation}
 By H\"older's inequality, we have
\begin{equation}\label{1E_kE_k+1_2}
\int_{\mathbb{R}^3}\nabla^{k} {\rm div}u\cdot\nabla^{k}{\rm div}(\varrho u)\,dx
\lesssim \sum_{0\le \ell\le k+1}\norm{ \nabla^\ell \varrho \nabla^{k+1-\ell} u}_{L^2}\norm{\nabla^{k+1}u}_{L^2}.
\end{equation}
If $0\le\ell\le \left[\frac{k+1}{2}\right]$, by Lemma \ref{1interpolation}, we have
\begin{equation}\label{bbb}
\begin{split}
\norm{ \nabla^\ell \varrho \nabla^{k+1-\ell} u}_{L^2}&\lesssim \norm{ \nabla^\ell \varrho }_{L^\infty} \norm{  \nabla^{k+1-\ell} u}_{L^2}
\\&\lesssim \norm{ \nabla^\alpha \varrho }_{L^2}^{1-\frac{\ell}{k+1}}\norm{ \nabla^{k+1} \varrho }_{L^2}^\frac{\ell}{k+1} \norm{   u}_{L^2}^\frac{\ell}{k+1} \norm{  \nabla^{k+1} u}_{L^2} ^{1-\frac{\ell}{k+1}} \\&\lesssim \sqrt{\mathcal{E}_0^3}\left(\norm{\nabla^{k+1}\varrho}_{L^2}+\norm{\nabla^{k+1}u}_{L^2}\right),
\end{split}
\end{equation}
where $\alpha$ is defined by
\begin{equation}
\begin{split}
&\frac{\ell}{3}=\left(\frac{\alpha}{3}-\frac{1}{2}\right)\times\left(1-\frac{\ell}{k+1}\right)
+\left(\frac{k+1}{3}-\frac{1}{2}\right)\times \frac{\ell}{k+1}
\\ &\quad \Longrightarrow \alpha=\frac{ 3(k+1)}{2(k+1-\ell)}\le 3 \, \text{ since }\ell\le \frac{k+2}{2}.
\end{split}
\end{equation}
While for $ \ell> \left[\frac{k+1}{2}\right]+1$ (then $k+1-\ell\le \left[\frac{k+1}{2}\right]$), we can then interchange the roles of $\varrho$ and $u$ to deduce that \eqref{bbb} holds also for this case. Thus, in view of \eqref{1E_kE_k+1_1}--\eqref{bbb}, we obtain
\begin{equation}\label{1E_kE_k+1_3}
\begin{split}
-\int_{\mathbb{R}^3} \nabla^{k} u_t\cdot\nabla\nabla^k\varrho\,dx \le-\frac{d}{dt}\int_{\mathbb{R}^3} \nabla^{k}
u\cdot\nabla\nabla^k\varrho\,dx+C\norm{\nabla^{k+1} u}_{L^2}^2+ C\sqrt{\mathcal{E}_0^3} \norm{\nabla^{k+1}\varrho}_{L^2}^2.
\end{split}
\end{equation}

Next, note that it has been already proved along the proof of Lemma \ref{1Ekle} that
\begin{equation}\label{1E_kE_k+1_4}
\begin{split}
&\norm{\nabla^{k}\left(u\cdot\nabla u+h(\varrho)( \mu \Delta u+( \mu + \lambda )\nabla{\rm div}
u)+f(\varrho)\nabla\varrho\right)}_{L^2}
\\&\quad\lesssim \sqrt{\mathcal{E}_0^3}
\left(\norm{\nabla^{k+1}\varrho}_{L^2}+\norm{\nabla^{k+2}u}_{L^2}\right).
\end{split}
\end{equation}

We now use the integration by parts and the Poisson equation $\eqref{1NS2}_3$ to have
\begin{equation}\label{1E_kE_k+1_0poi}
-\int_{\mathbb{R}^3}  \nabla^k\nabla\Phi\cdot \nabla\nabla^k\varrho\,dx =\int_{\mathbb{R}^3}  \nabla^k\Delta\Phi\nabla^k\varrho\,dx
=\int_{\mathbb{R}^3}  |\nabla^k\varrho|^2\,dx.
\end{equation}
On the other hand, it follows from the Poisson equation that
\begin{equation}\label{p es2}
\norm{\nabla^{k+1}\nabla\Phi}_{L^2}^2=\norm{\nabla^{k}\Delta\Phi}_{L^2}^2= \norm{\nabla^{k}\varrho}_{L^2}^2\text{ and }\norm{\nabla^{k+2}\nabla\Phi}_{L^2}^2= \norm{\nabla^{k+1}\varrho}_{L^2}^2.
\end{equation}

Consequently, by \eqref{1E_kE_k+1_3}--\eqref{p es2}, together with Cauchy's inequality, since $\sqrt{\mathcal{E}_0^3}\le \delta$ is small, we then deduce
\eqref{1E_kE_k+1} from \eqref{1E_kE_k+1_0}.
\end{proof}

%%%%%%%%%%%%%%%%%%%%%%%%%%%%%%%%%%%%%%%%%%%%%%%
\section{Negative Sobolev estimates}
%%%%%%%%%%%%%%%%%%%%%%%%%%%%%%%%%%%%%%%%%%%%%%%
In this section, we will derive the evolution of the negative Sobolev norms of the solution to \eqref{1NS2}. In order to estimate the nonlinear terms, we need to restrict ourselves to that $s\in (0,3/2)$. We will establish the following lemma.
\begin{lemma}\label{1Esle}
If $\sqrt{\mathcal{E}_0^3(t)}\le \delta$, then for $s\in(0,1/2]$, we have
\begin{equation}\label{1E_s}
\begin{split}&
\frac{d}{dt}\int_{\mathbb{R}^3}  |\Lambda^{-s} \varrho|^2+|\Lambda^{-s} u|^2+ |\Lambda^{-s}\nabla\Phi|^2\,dx
+C\norm{\nabla\Lambda^{-s} u}_{L^2}^2
\\&\quad\lesssim \left(\norm{\varrho}_{H^2}^2 +\norm{ \nabla u}_{H^1}^2 \right) \left(\norm{\Lambda^{-s}\varrho}_{L^2}+\norm{\Lambda^{-s}u}_{L^2}+\norm{\Lambda^{-s}\nabla\Phi}_{L^2}\right);
\end{split}
\end{equation}
and for $s\in(1/2,3/2)$, we have
\begin{equation}\label{1E_s2}
\begin{split}
&\frac{d}{dt}\int_{\mathbb{R}^3}  |\Lambda^{-s}\varrho|^2+|\Lambda^{-s} u|^2+ |\Lambda^{-s}\nabla\Phi|^2\,dx
+C\norm{\nabla\Lambda^{-s} u}_{L^2}^2
\\&\quad\lesssim \norm{ (\varrho, u)}_{L^2}^{s-1/2}\left(\norm{\varrho}_{H^2}+\norm{\nabla u}_{H^1}\right)^{5/2-s}\left(\norm{\Lambda^{-s}\varrho}_{L^2}+\norm{\Lambda^{-s}u}_{L^2}+\norm{\Lambda^{-s}\nabla\Phi}_{L^2}\right).
\end{split}
\end{equation}
\end{lemma}

\begin{proof}
Applying $\Lambda^{-s}$  to $\eqref{1NS2}_1, \eqref{1NS2}_2$ and multiplying the resulting identities by $\displaystyle \Lambda^{-s}\varrho, \Lambda^{-s}
u$ respectively, summing up them and then integrating over $\mathbb{R}^3$ by parts, we obtain
\begin{equation}\label{1E_s_0}
\begin{split}
&\frac{1}{2}\frac{d}{dt}\int_{\mathbb{R}^3}  |\Lambda^{-s} \varrho|^2+|  \Lambda^{-s} u|^2\,dx +\int_{\mathbb{R}^3} \mu |\nabla \Lambda^{-s}
u|^2+( \mu + \lambda )| {\rm div} \Lambda^{-s} u|^2\,dx -\int_{\mathbb{R}^3}\Lambda^{-s}\nabla\Phi\cdot \Lambda^{-s}u\,dx
\\&\quad=\int_{\mathbb{R}^3} \Lambda^{-s}\left(-\varrho{\rm div}u-u\cdot\nabla\varrho\right)\Lambda^{-s}\varrho
-\Lambda^{-s}\left(u\cdot\nabla u+h(\varrho)( \mu \Delta u+( \mu + \lambda )\nabla{\rm div}u)+f(\varrho)\nabla\varrho\right)\cdot\Lambda^{-s} u\,dx
\\&\quad:=W_1+W_2+W_3+W_4+W_5.
\end{split}
\end{equation}

In order to estimate the nonlinear terms in the right hand side of \eqref{1E_s_0}, we shall use the estimate \eqref{1Riesz es} of Riesz potential in Lemma \ref{1Riesz}. This forces us to require that $s\in (0,3/2)$. If $ s\in(0,1/2]$, then $1/2+s/3<1$ and $3/s\ge 6$. Then by Lemma
\ref{1Riesz} and Lemma \ref{1interpolation}, together with H\"older's and Young's inequalities, we have
\begin{equation}\label{1E_s_1}
\begin{split}
W_1&=- \int_{\mathbb{R}^3} \Lambda^{-s}(\varrho{\rm div}u )\Lambda^{-s}\varrho\,dx
\lesssim\norm{\Lambda^{-s}(\varrho{\rm div}u )}_{L^2}\norm{\Lambda^{-s}\varrho}_{L^2}
\\&\lesssim \norm{\varrho{\rm div}u }_{L^\frac{1}{1/2+s/3}}\norm{\Lambda^{-s}\varrho}_{L^2}
\lesssim \norm{\varrho}_{L^{3/s}}\norm{\nabla u }_{L^2}\norm{\Lambda^{-s}\varrho}_{L^2}
\\&\lesssim\norm{ \nabla\varrho}_{L^2}^{1/2-s}\norm{ \nabla^2\varrho}_{L^2}^{1/2+s}\norm{\nabla u}_{L^2}\norm{\Lambda^{-s}\varrho}_{L^2}
\\&\lesssim\left(\norm{ \nabla\varrho}_{H^1}^2+\norm{\nabla u}_{L^2}^2\right)\norm{\Lambda^{-s}\varrho}_{L^2}.
\end{split}
\end{equation}
Similarly, we can bound the remaining terms by
\begin{eqnarray}
\label{1E_s_2} &&W_2 =- \int_{\mathbb{R}^3}\Lambda^{-s}(u\cdot\nabla\varrho)\Lambda^{-s}\varrho\,dx \lesssim \left(\norm{ \nabla u}_{H^1}^2+\norm{
\nabla\varrho}_{L^2}^2\right)\norm{\Lambda^{-s}\varrho}_{L^2},
\\\label{1E_s_3}&& W_3=-\int_{\mathbb{R}^3}\Lambda^{-s}\left(u\cdot\nabla
u \right)\cdot\Lambda^{-s} u\,dx \lesssim  \left( \norm{ \nabla u}_{H^1}^2+\norm{\nabla^2 u}_{L^2}^2\right)\norm{\Lambda^{-s}u}_{L^2},
\\&&
W_4=-\int_{\mathbb{R}^3} \Lambda^{-s}\left( h(\varrho)( \mu \Delta u+( \mu + \lambda )\nabla{\rm div} u) \right) \Lambda^{-s} u\,dx
\\&&\quad\ \ \lesssim \left(\norm{ \nabla\varrho}_{H^1}^2+\norm{\nabla^2 u }_{L^2}^2\right)\norm{\Lambda^{-s}u}_{L^2},
\\\label{1E_s_5}&& W_5=-\int_{\mathbb{R}^3} \Lambda^{-s}\left( f(\varrho)\nabla\varrho\right)\cdot\Lambda^{-s} u\,dx \lesssim \left(\norm{
\nabla\varrho}_{H^1}^2+\norm{\nabla^2\varrho}_{L^2}^2\right)\norm{\Lambda^{-s}u}_{L^2}.
\end{eqnarray}

Now if $s\in(1/2,3/2)$, we shall estimate the right hand side of \eqref{1E_s_0} in a different way. Since $ s\in(1/2,3/2)$, we
have that $1/2+s/3<1$ and $2<3/s< 6$. Then by Lemma \ref{1Riesz} and Lemma \ref{1interpolation}, we obtain
\begin{equation}\label{1E_s_12}
\begin{split}
W_1&=- \int_{\mathbb{R}^3} \Lambda^{-s}(\varrho{\rm div}u )\Lambda^{-s}\varrho\,dx
\lesssim \norm{\Lambda^{-s}(\varrho{\rm div}u )}_{L^2}\norm{\Lambda^{-s}\varrho}_{L^2}
\\&\lesssim \norm{\varrho{\rm div}u }_{L^\frac{1}{1/2+s/3}}\norm{\Lambda^{-s}\varrho}_{L^2}
\lesssim \norm{\varrho}_{L^{3/s}}\norm{\nabla u }_{L^2}\norm{\Lambda^{-s}\varrho}_{L^2}
\\&\lesssim\norm{ \varrho}_{L^2}^{s-1/2 }\norm{ \nabla\varrho}_{L^2}^{3/2-s}\norm{\nabla u}_{L^2}\norm{\Lambda^{-s}\varrho}_{L^2}.
\end{split}
\end{equation}
Similarly, we can bound the remaining terms by
\begin{eqnarray}
\label{1E_s_22}&& W_2=- \int_{\mathbb{R}^3}\Lambda^{-s}(u\cdot\nabla\varrho)\Lambda^{-s}\varrho\,dx \lesssim \norm{ u}_{L^2}^{s-1/2 }\norm{
\nabla u}_{L^2}^{3/2-s}\norm{\nabla \varrho}_{L^2}\norm{\Lambda^{-s}\varrho}_{L^2},
\\\label{1E_s_32}&&
W_3=-\int_{\mathbb{R}^3}\Lambda^{-s}\left(u\cdot\nabla u \right)\cdot\Lambda^{-s} u\,dx \lesssim \norm{ u}_{L^2}^{s-1/2 }\norm{ \nabla
u}_{L^2}^{3/2-s}\norm{\nabla u}_{L^2}\norm{\Lambda^{-s}u}_{L^2},
\\\nonumber&&
W_4=-\int_{\mathbb{R}^3}\Lambda^{-s}\left( h(\varrho)( \mu \Delta u+( \mu + \lambda )\nabla{\rm div} u) \right) \Lambda^{-s} u\,dx
\\\label{1E_s_42}&&\quad\ \,\lesssim   \norm{ \varrho}_{L^2}^{s-1/2 }\norm{ \nabla \varrho}_{L^2}^{3/2-s}\norm{\nabla^2 u }_{L^2}\norm{\Lambda^{-s}u}_{L^2},
\\\label{1E_s_52} &&W_5=-\int_{\mathbb{R}^3} \Lambda^{-s}\left( f(\varrho)\nabla\varrho\right)\cdot\Lambda^{-s} u\,dx \lesssim \norm{
\varrho}_{L^2}^{s-1/2 }\norm{ \nabla \varrho}_{L^2}^{3/2-s}\norm{\nabla \varrho}_{L^2}\norm{\Lambda^{-s}\varrho}_{L^2}.
\end{eqnarray}

Finally, we  turn to the left hand side of \eqref{1E_s_0}. For the second term, we have
\begin{equation}
\int_{\mathbb{R}^3} \mu |\nabla \Lambda^{-s}
u|^2+( \mu + \lambda )| {\rm div} \Lambda^{-s} u|^2\,dx -\int_{\mathbb{R}^3}\Lambda^{-s}\nabla\Phi\cdot \Lambda^{-s}u\,dx
\ge \sigma_0\norm{\nabla  \Lambda^{-s}
u}_{L^2}^2.
\end{equation}
While for the Poisson term, by the continuity equation $\eqref{1NS2}_1$ and the Poisson equation
$\eqref{1NS2}_3$ and the integration by parts, we get
\begin{equation}\label{1E_s_0poi}
\begin{split}
&-\int_{\mathbb{R}^3}\Lambda^{-s}\nabla\Phi\cdot \Lambda^{-s}u\,dx =\int_{\mathbb{R}^3}\Lambda^{-s}\Phi \Lambda^{-s}{\rm div }u\,dx
\\&\quad= \int_{\mathbb{R}^3}-\Lambda^{-s}\Phi \Lambda^{-s}\partial_t\varrho-\Lambda^{-s}\Phi \Lambda^{-s}{\rm div }(\varrho u)\,dx
\\&\quad= \int_{\mathbb{R}^3}-\Lambda^{-s}\Phi \Lambda^{-s}\partial_t\Delta\Phi+\Lambda^{-s}\nabla\Phi \cdot\Lambda^{-s}(\varrho u)\,dx
\\&\quad=\frac{1}{21}\frac{d}{dt}\int_{\mathbb{R}^3}|\Lambda^{-s}\nabla\Phi|^2\,dx+ \int_{\mathbb{R}^3}\Lambda^{-s}\nabla\Phi \cdot\Lambda^{-s}(\varrho u)\,dx.
\end{split}
\end{equation}
If $s\in (0,1/2]$, we use Lemma \ref{1Riesz} and Lemma \ref{1interpolation} to obtain
\begin{equation}\label{1E_s_0poi1}
\norm{\Lambda^{-s}(\varrho u)}_{L^2}\lesssim \norm{\varrho }_{L^2}\norm{u }_{L^{3/s}}\lesssim \norm{\varrho }_{L^2}\norm{\nabla u
}_{L^2}^{1/2-s}\norm{\nabla^2 u }_{L^2}^{1/2+s};
\end{equation}
and if $s\in (1/2,3/2)$, we have
\begin{equation}\label{1E_s_0poi2}
\norm{\Lambda^{-s}(\varrho u)}_{L^2}\lesssim \norm{\varrho }_{L^2}\norm{u }_{L^{3/s}}\lesssim \norm{\varrho }_{L^2}\norm{ u
}_{L^2}^{s-1/2}\norm{\nabla u }_{L^2}^{3/2-s};
\end{equation}

Consequently, in light of \eqref{1E_s_1}--\eqref{1E_s_0poi2}, we deduce \eqref{1E_s2} from \eqref{1E_s_0}.
\end{proof}

%%%%%%%%%%%%%%%%%%%%%%%%%%%%%%%%%%%%%%%%%%%%%%%
\section{Proof of Theorem \ref{1mainth}}
%%%%%%%%%%%%%%%%%%%%%%%%%%%%%%%%%%%%%%%%%%%%%%%

In this section, we shall combine all the energy estimates that we have derived in the previous two sections and the Sobolev interpolation to prove Theorem \ref{1mainth}.

We first close the energy estimates at each $\ell$-th level in our weak sense to prove \eqref{1HNbound}. Let $N\ge 3$ and $0\le\ell\le m-1$ with $1\le m\le N$. Summing up the estimates \eqref{1E_k} of Lemma \ref{1Ekle} for from $k=\ell$ to $m$, since $\sqrt{\mathcal{E}_0^3}\le \delta$ is small, we obtain
\begin{equation}\label{1proof1}
\begin{split}
&\frac{d}{dt}\sum_{\ell\le k\le m}\left( \norm{\nabla^{k}
\varrho}_{L^2}^2+\norm{\nabla^{k}u}_{L^2}^2+ \norm{\nabla^k\nabla\Phi}_{L^2}^2\right) +C_1\sum_{\ell+1\le k\le
m+1}\norm{\nabla^{k} u}_{L^2}^2
\\&\quad\le C_2 \delta\left(\sum_{\ell\le k\le m} \norm{\nabla^{k}\varrho}_{L^2}^2+\sum_{\ell+1\le k\le
m+1}\norm{\nabla^{k}\nabla\Phi}_{L^2}^2\right).
\end{split}
\end{equation}
Summing up the estimates \eqref{1E_kE_k+1} of Lemma \ref{1EkEk+1le} for from $k=\ell$ to $m-1$, we have
\begin{equation}\label{1proof2}
\begin{split}
&\frac{d}{dt}\sum_{\ell\le k\le m-1}\int_{\mathbb{R}^3}  \nabla^ku\cdot\nabla\nabla^k\varrho\,dx +C_3\left(\sum_{\ell\le k\le
m}\norm{\nabla^{k}\varrho}_{L^2}^2+\sum_{\ell+1\le k\le m+1}\norm{\nabla^{k}\nabla\Phi}_{L^2}^2\right)\\ &\quad\le C_4 \sum_{\ell+1\le k\le
m+1}\norm{\nabla^{k}u}_{L^2}^2.
\end{split}
\end{equation}
Multiplying \eqref{1proof2} by $2C_2\delta/C_3$, adding the resulting inequality with \eqref{1proof1}, since $\delta>0$ is small, we deduce that there exists a constant
$C_5>0$ such that for $0\le \ell\le m-1$,
\begin{equation}\label{1proof3}
\begin{split}
& \frac{d}{dt}\left\{\sum_{\ell\le k\le m} \left( \norm{\nabla^{k} \varrho}_{L^2}^2+\norm{\nabla^{k}
u}_{L^2}^2+ \norm{\nabla^k\nabla\Phi}_{L^2}^2\right) +\frac{2C_2\delta}{C_3}\sum_{\ell\le k\le m-1}\int_{\mathbb{R}^3}
\nabla^ku\cdot\nabla\nabla^k\varrho\,dx\right\}
\\&\quad+C_5\left\{\sum_{\ell\le k\le
m}\norm{\nabla^{k}\varrho}_{L^2}^2+\sum_{\ell+1\le k\le m+1}\norm{\nabla^{k}
u}_{L^2}^2+\sum_{\ell+1\le k\le m+1}\norm{\nabla^{k}\nabla\Phi}_{L^2}^2\right\}\le 0.
\end{split}
\end{equation}
We define $\mathcal{E}_\ell^m(t)$ to be $C_5^{-1}$ times the expression under the time derivative in \eqref{1proof3}. Observe that since
$\delta$ is small, $\mathcal{E}_\ell^m(t)$ is equivalent to $\norm{\nabla^\ell
\varrho(t)}_{H^{m-\ell}}^2+\norm{\nabla^\ell u(t)}_{H^{m-\ell}}^2+\norm{\nabla^\ell \nabla\Phi(t)}_{H^{m-\ell}}^2$.
Then we may write
\eqref{1proof3} as that for $0\le \ell\le m-1$,
\begin{equation}\label{1proof5}
\frac{d}{dt}\mathcal{E}_\ell^m+\norm{\nabla^{\ell}\varrho}_{H^{m-\ell}}^2+\norm{\nabla^{\ell+1}u}_{H^{m-\ell}}^2+\norm{\nabla^{\ell+1}\nabla\Phi}_{H^{m-\ell}}^2\le
0.
\end{equation}

Now taking $\ell=0$ and $m=3$ in \eqref{1proof5} and then integrating directly in time, we get
\begin{equation}\label{1proof6}
\begin{split}
\norm{ \varrho(t)}_{3}^2+\norm{ u(t)}_{H^{3}}^2+\norm{ \nabla\Phi(t)}_{H^{3}}^2 \lesssim \mathcal{E}_0^{3}(t) \le
\mathcal{E}_0^{3}(0) \lesssim \norm{ \varrho_0}_{3}^2+\norm{ u_0}_{H^{3}}^2+\norm{ \nabla\Phi_0}_{H^{3}}^2.
\end{split}
\end{equation}
By a standard continuity argument, this closes the a priori estimates \eqref{1a priori} if at the initial time we assume that $\norm{ \varrho_0}_{3}^2+\norm{
u_0}_{H^{3}}^2+\norm{ \nabla\Phi_0}_{H^{3}}^2\le \delta_0$ is sufficiently small. This in turn allows us to take $\ell=0$ and $m=N$ in
\eqref{1proof5}, and then integrate it directly in time to obtain  \eqref{1HNbound}.

Next, we turn to prove \eqref{1H-sbound}--\eqref{1decay2}. However, we are not able to prove them for all $s\in[0,3/2)$ at this moment. We shall first
prove them for $s\in [0,1/2]$.
\begin{proof}[Proof of \eqref{1H-sbound}--\eqref{1decay2} for {$s\in[0,1/2]  $}]
Define $\mathcal{E}_{-s}(t):=\norm{ \Lambda^{-s}\varrho(t)}_{L^2}^2+\norm{\Lambda^{-s} u(t)}_{L^2}^2+\norm{\Lambda^{-s} \nabla\Phi(t)}_{L^2}^2$. Then, integrating in time \eqref{1E_s}, by the bound  \eqref{1HNbound}, we obtain that for $s\in(0,1/2]$,
\begin{equation}\label{1-sin2}
\begin{split}
\mathcal{E}_{-s}(t)
&\le {\mathcal{E}_{-s}(0)}+C\int_0^t \left(\norm{\varrho}_{H^2}^2 +\norm{ \nabla u}_{H^1}^2 \right)\sqrt{\mathcal{E}_{-s}(\tau)} \,d\tau
\\&\le C_0\left(1+\sup_{0\le\tau\le t}\sqrt{\mathcal{E}_{-s}(\tau)}\right).
\end{split}
\end{equation}
This implies \eqref{1H-sbound} for $s\in[0,1/2]$, that is,
\begin{equation}\label{1proof8}
\norm{ \Lambda^{-s}\varrho(t)}_{L^2}^2+\norm{\Lambda^{-s} u(t)}_{L^2}^2+\norm{\Lambda^{-s} \nabla\Phi(t)}_{L^2}^2 \le C_0\ \hbox{ for }s\in[0,1/2].
\end{equation}
If $\ell=1,\dots, N-1$, we may use Lemma \ref{1-sinte} to  have
\begin{equation}\label{1proof10}
\norm{\nabla^{\ell+1} f}_{L^2}
\ge C  \norm{ \Lambda^{-s} f}_{L^2}^{-\frac{1}{\ell+s}}\norm{\nabla^\ell f}_{L^2}^{1+\frac{1}{\ell+s}}.
\end{equation}
By this fact and \eqref{1proof8}, we may find
\begin{equation}
\norm{\nabla^{\ell+1}u}_{L^2}^2+\norm{\nabla^{\ell+1}\nabla\Phi}_{L^2}^2
\ge C_0 \left(\norm{\nabla^{\ell}u}_{L^2}^2+\norm{\nabla^{\ell}\nabla\Phi}_{L^2}^2\right)^{1+\frac{1}{\ell+s}}.
\end{equation}
This together with \eqref{1HNbound} implies in particular that for $\ell=1,\dots,N-1$,
\begin{equation}\label{1inin}
\begin{split}
&\norm{\nabla^{\ell}\varrho}_{H^{N-\ell}}^2+\norm{\nabla^{\ell+1}u}_{H^{N-\ell-1}}^2+\norm{\nabla^{\ell+1}\nabla\Phi}_{H^{N-\ell-1}}^2
\\&\quad\ge C_0\left(\norm{\nabla^{\ell}\varrho}_{H^{N-\ell}}^2+\norm{\nabla^{\ell}u}_{H^{N-\ell}}^2+\norm{\nabla^{\ell}\nabla\Phi}_{H^{N-\ell}}^2\right)^{1+\frac{1}{\ell+s}}.
\end{split}
\end{equation}
Thus, we deduce from \eqref{1proof5} with $m=N$ the following time differential inequality
\begin{equation}\label{1proof11}
\frac{d}{dt}\mathcal{E}_\ell^N+C_0\left(\mathcal{E}_\ell^N\right)^{1+\frac{1}{\ell+s}}\le 0\ \hbox{ for }\ell=1,\dots,N-1.
\end{equation}
Solving this inequality directly gives
\begin{equation}\label{1proof12}
\mathcal{E}_\ell^N(t)\le C_0(1+t)^{-(\ell+s)}\ \hbox{ for }\ell=1,\dots,4N-1.
\end{equation}
This implies that for $s\in[0,1/2]$,
\begin{equation}\label{1proof13}
\norm{\nabla^\ell \varrho(t)}_{H^{N-\ell}}^2+\norm{\nabla^\ell u(t)}_{H^{N-\ell}}^2+\norm{\nabla^\ell \nabla\Phi(t)}_{H^{N+1-\ell}}^2\le
C_0(1+t)^{-(\ell+s)}\ \hbox{ for }\ell=1,\dots,N-1.
\end{equation}
On the other hand, since $\varrho={\rm div}\nabla\Phi$, we have
\begin{equation}\label{1proof132}
\norm{\nabla^\ell \varrho(t)}_{L^2}^2\le \norm{\nabla^{\ell+1} \nabla\Phi(t)}_{L^2}^2\le C_0(1+t)^{-(\ell+1+s)}\ \hbox{ for }\ell=0,\dots,N-2.
\end{equation}
Hence, by \eqref{1proof13}, \eqref{1proof132}, \eqref{1proof8} and the interpolation, we get
\eqref{1decay}--\eqref{1decay2} for $s\in[0,1/2]$.
\end{proof}

Now we can present the
\begin{proof}[Proof of \eqref{1H-sbound}--\eqref{1decay2} for {$s\in(1/2,3/2)$}]
Notice that the arguments for the case $s\in[0,1/2]$ can not be applied to this case. However, observing that we have $\varrho_0, u
_0,\nabla\Phi_0\in \dot{H}^{-1/2}$ since $\dot{H}^{-s}\cap L^2\subset\dot{H}^{-s'}$ for any $s'\in [0,s]$, we then deduce from what we have proved
for \eqref{1H-sbound}--\eqref{1decay2} with $s=1/2$ that the following decay result holds:
\begin{equation}\label{1proof14}
\norm{\nabla^\ell \varrho(t)}_{H^{N-\ell}}^2+\norm{\nabla^\ell u(t)}_{H^{N-\ell}}^2+\norm{\nabla^\ell \nabla\Phi(t)}_{H^{N+1-\ell}}^2 \le
C_0(1+t)^{-(\ell+1/2)}\ \hbox{ for }\ell=0,\dots, N-1.
\end{equation}
and
\begin{equation}\label{1proof1311}
\norm{\nabla^\ell \varrho(t)}_{L^2}^2\le  C_0(1+t)^{-(\ell+3/2)}\ \hbox{ for }\ell=0,\dots,N-2.
\end{equation}
Hence, by \eqref{1proof14}--\eqref{1proof1311}, we deduce from \eqref{1E_s2} that for $s\in(1/2,3/2)$,
\begin{equation}\label{1-sin2''}
\begin{split}
\mathcal{E}_{-s}(t)
&\le {\mathcal{E}_{-s}(0)}+C\int_0^t \norm{ (\varrho, u)}_{L^2}^{s-1/2}\left(\norm{\varrho}_{H^2}+\norm{\nabla u}_{H^1}\right)^{5/2-s}\sqrt{\mathcal{E}_{-s}(\tau)} \,d\tau
\\&\le C_0+ C_0\int_0^t(1+\tau)^{-(7/4-s/2)}\,d\tau\sup_{0\le\tau\le t}\sqrt{\mathcal{E}_{-s}(\tau)}
\\&\le C_0\left(1+\sup_{0\le\tau\le t}\sqrt{\mathcal{E}_{-s}(\tau)}\right).
\end{split}
\end{equation}
This implies \eqref{1H-sbound} for $s\in(1/2,3/2)$, that is,
\begin{equation}\label{1sss}
\norm{ \Lambda^{-s}\varrho(t)}_{L^2}^2+\norm{\Lambda^{-s} u(t)}_{L^2}^2+\norm{\Lambda^{-s} \nabla\Phi(t)}_{L^2}^2 \le C_0\ \hbox{ for }s\in(1/2,3/2).
\end{equation}

Now that we have proved \eqref{1sss}, we may repeat the arguments leading to \eqref{1decay}--\eqref{1decay2} for $s\in[0,1/2]$ to prove that they hold  also for $s\in(1/2,3/2)$.
\end{proof}

\appendix

%%%%%%%%%%%%%%%%%%%%%%%%%%%%%%%%%%%%%%%%%%%%%%%
\section{Analytic tools}\label{1section_appendix}
%%%%%%%%%%%%%%%%%%%%%%%%%%%%%%%%%%%%%%%%%%%%%%%

%%%%%%%%%%%%%%%%%%%%%%%%%%%%%%%%%%%%%%%%%%%%%%%
\subsection{Sobolev type inequalities}
%%%%%%%%%%%%%%%%%%%%%%%%%%%%%%%%%%%%%%%%%%%%%%%

We will extensively use the Sobolev interpolation of the Gagliardo-Nirenberg inequality.
 \begin{lemma}\label{1interpolation}
 Let $0\le m, \alpha\le \ell$, then we have
\begin{equation}
\norm{\nabla^\alpha f}_{L^p}\lesssim \norm{  \nabla^mf}_{L^q}^{1-\theta}\norm{ \nabla^\ell f}_{L^r}^{\theta}
\end{equation}
where $\alpha$ satisfies
\begin{equation}
\frac{\alpha}{3}-\frac{1}{p}=\left(\frac{m}{3}-\frac{1}{q}\right)(1-\theta)+\left(\frac{\ell}{3}-\frac{1}{r}\right)\theta.
\end{equation}
\begin{proof}
This is a special case of  \cite[pp. 125, THEOREM]{N1959}.
\end{proof}
\end{lemma}

 Next, to estimate the $L^\infty$ norm of the spatial derivatives of $h$ and $f$ defined by \eqref{1hf}, we shall record the following estimate:
\begin{lemma}\label{infty}
Assume that $\norm{\varrho}_{L^\infty}\le 1$. Let $g(\varrho)$ be a smooth function of $\varrho$ with bounded derivatives of any order, then  for any integer $m\ge1$  we have
\begin{equation}
\norm{\nabla^m(g(\varrho))}_{L^\infty} \lesssim \norm{\nabla^{m}\varrho}_{L^\infty}.
\end{equation}
\end{lemma}
\begin{proof}
Notice that for $m\ge 1$,
\begin{equation}
\nabla^m(g(\varrho))=\hbox{ a sum of products }g^{\gamma_1,\dots,\gamma_n}(\varrho)\nabla^{\gamma_1}\varrho\cdots\nabla^{\gamma_n}\varrho,
\end{equation}
where the functions $g^{\gamma_1,\dots,\gamma_n}(\varrho)$ are some derivatives of $g(\varrho)$ and $1\le \gamma_i\le m,\ i=1,\dots,n$ with $\gamma_1+\cdots+\gamma_n=m$. We then use the Sobolev interpolation of Lemma \ref{1interpolation} to bound
\begin{equation}
\begin{split}
\norm{\nabla^m(g(\varrho))}_{L^\infty}
 &\lesssim
 \norm{\nabla^{\gamma_1}\varrho }_{L^\infty}\cdots \norm{ \nabla^{\gamma_n}\varrho}_{L^\infty}
 \\&\lesssim
 \norm{ \varrho}_{L^\infty}^{1-\gamma_1/m}\norm{\nabla^{m}\varrho}_{L^\infty}^{\gamma_1/m}\cdots\norm{ \varrho}_{L^\infty}^{1-\gamma_n/m}\norm{\nabla^{m}\varrho}_{L^\infty}^{\gamma_n/m}
 \lesssim\norm{\varrho}_{L^\infty}^{n-1}\norm{\nabla^m\varrho}_{L^\infty}.
\end{split}
\end{equation}
Hence, we conclude our lemma since $\norm{\varrho}_{L^\infty}\le 1$.
\end{proof}

We recall the following commutator estimate:
\begin{lemma}\label{1commutator}
Let $m\ge 1$ be an integer and  define the commutator
\begin{equation}\label{1commuta}
[\nabla^m,f]g=\nabla^m(fg)-f\nabla^mg.
\end{equation}
Then we have
\begin{equation}
\norm{[\nabla^m,f]g}_{L^p} \lesssim \norm{\nabla f}_{L^{p_1}}\norm{\nabla^{m-1}g}_{L^{p_2}}+\norm{\nabla^m f}_{L^{p_3}}\norm{ g}_{L^{p_4}},
\end{equation}
where $p,p_2,p_3\in(1,+\infty)$ and
\begin{equation}
\frac{1}{p}=\frac{1}{p_1}+\frac{1}{p_2}=\frac{1}{p_3}+\frac{1}{p_4}.
\end{equation}
\end{lemma}
\begin{proof}
For $p=p_2=p_3=2$, it can be proved by using Lemma \ref{1interpolation}. For the general cases, one
may refer to \cite[Lemma 3.1]{J2004}
\end{proof}

%%%%%%%%%%%%%%%%%%%%%%%%%%%%%%%%%%%%%%%%%%%%%%%
\subsection{Negative Sobolev norms}
%%%%%%%%%%%%%%%%%%%%%%%%%%%%%%%%%%%%%%%%%%%%%%%

We define the operator $\Lambda^s, s\in \mathbb{R}$ by
\begin{equation}\label{1Lambdas}
\Lambda^s f(x)=\int_{\mathbb{R}^3}|\xi|^s\hat{f}(\xi)e^{2\pi ix\cdot\xi}\,d\xi,
\end{equation}
where $\hat{f}$ is the  Fourier transform of $f$. We define the homogeneous  Sobolev space
$\dot{H}^s$ of all $f$ for which $\norm{f}_{\dot{H}^s}$  is finite, where
\begin{equation}\label{1snorm}
\norm{f}_{\dot{H}^s}:=\norm{\Lambda^s f}_{L^2}=\norm{|\xi|^s \hat{f}}_{L^2}.
\end{equation}
We will use the non-positive index $s$. For convenience, we will change the index to be ``$-s$" with $s\ge 0$. We will employ the following special Sobolev interpolation:
\begin{lemma}\label{1-sinte}
Let $s\ge 0$ and $\ell\ge 0$, then we have
\begin{equation}\label{1-sinterpolation}
\norm{\nabla^\ell f}_{L^2}\le \norm{\nabla^{\ell+1} f}_{L^2}^{1-\theta}\norm{f}_{\Dot{H}^{-s}}^\theta, \hbox{ where }\theta=\frac{1}{\ell+1+s}.
\end{equation}
\end{lemma}
\begin{proof}
By the Parseval theorem, the definition of \eqref{1snorm} and H\"older's inequality, we have
\begin{equation}
\norm{\nabla^\ell f}_{L^2}
=\norm{|\xi|^\ell \hat{f}}_{L^2}\le  \norm{|\xi|^{\ell+1} \hat{f}}_{L^2}^{1-\theta}\norm{|\xi|^{-s} \hat{f}}_{L^2}^\theta
=\norm{\nabla^{\ell+1}f}_{L^2}^{1-\theta}\norm{ f}_{\Dot{H}^{-s}}^\theta.
\end{equation}
\end{proof}

If $s\in(0,3)$, $\Lambda^{-s}f$ defined by \eqref{1Lambdas} is the Riesz potential. The Hardy-Littlewood-Sobolev theorem implies the following $L^p$ type
inequality for the Riesz potential:
\begin{lemma}\label{1Riesz}
Let $0<s<3,\ 1<p<q<\infty,\ 1/q+s/3=1/p$, then
\begin{equation}\label{1Riesz es}
\norm{\Lambda^{-s}f}_{L^q}\lesssim\norm{ f}_{L^p}.
\end{equation}
\end{lemma}
\begin{proof}
See \cite[pp. 119, Theorem 1]{S}.
\end{proof}

\end{document}